\RequirePackage[l2tabu, orthodox]{nag}
\documentclass[10pt]{article}%

\usepackage{amssymb}
\usepackage{amsfonts}
\usepackage{amsmath}
\usepackage{mathtools}
\usepackage{dsfont}

\usepackage{amsthm}
\usepackage{hyperref}
\usepackage{cleveref}

\usepackage{enumitem}

\usepackage{float}
\usepackage{graphicx}
\usepackage{caption}
\usepackage{epstopdf} 
\usepackage[all]{xy} 
\usepackage{tikz}

\usepackage[titletoc, title]{appendix} 
\usepackage{titling} 
\usepackage{url} 
\usepackage{color}
\usepackage{enumerate} 

\usepackage[latin1]{inputenc}
\usepackage[a4paper]{geometry}
\usepackage{microtype}

\providecommand{\U}[1]{\protect\rule{.1in}{.1in}}
\newtheorem{theo}{Theorem}[section]

\newtheorem{lem}[theo]{Lemma}
\newtheorem{cor}[theo]{Corollary}

\newtheorem{rem}[theo]{Remark}

\numberwithin{equation}{section}

\newcommand{\CC}{\mathbb{C}}

\newcommand{\EE}{\mathbb{E}}

\newcommand{\NN}{\mathbb{N}}

\newcommand{\PP}{\mathbb{P}}
\newcommand{\QQ}{\mathbb{Q}}
\newcommand{\RR}{\mathbb{R}}

\newcommand{\ZZ}{\mathbb{Z}}

\newcommand{\Ca}{ {\mathcal C }}

\newcommand{\Ra}{ {\mathcal R }}

\newcommand{\Fa}{ {\mathcal F }}
\newcommand{\Ga}{ {\mathcal G }}

\newcommand{\hide}[1]{}

\title{Tail universality of critical Gaussian multiplicative chaos}
\author{Mo Dick Wong\\
\\
Mathematical Institute, University of Oxford}
\date{\today}
\begin{document}
\maketitle

\abstract{In this article we study the tail probability of the mass of critical Gaussian multiplicative chaos (GMC) associated to a general class of log-correlated Gaussian fields in any dimension, including the Gaussian free field (GFF) in dimension two. More precisely, we derive a fully explicit formula for the leading order asymptotics for the tail probability and demonstrate a new universality phenomenon. Our analysis here shares similar philosophy with the subcritical case but requires a different approach due to complications in the analogous localisation step, and we also employ techniques from recent studies of fusion estimates in GMC theory.}

\tableofcontents

\section{Introduction} \label{sec:Intro}
Let $X(\cdot)$ be a log-correlated Gaussian field\footnote{All Gaussian fields/random variables in this article are centred unless otherwise specified.} on some bounded domain $D \subset \RR^d$ with covariance
\begin{align}\label{eq:cov}
\EE[X(x) X(y)] = -\log|x-y| + f(x, y), \qquad \forall x, y \in D.
\end{align}

\noindent The associated Gaussian multiplicative chaos (GMC) is a random measure formally written as the exponentiation of the underlying field, i.e.
\begin{align*}
M_{\gamma}(dx) = e^{\gamma X(x) - \frac{\gamma^2}{2} \EE[X(x)^2]}dx
\end{align*}

\noindent where $\gamma \in \RR$ is an intermittency parameter. First introduced by Kahane \cite{Kah1985} in an attempt to provide a mathematical framework for Kolmogorov-Obukhov-Mandelbrot's model of turbulence, the theory of GMCs has attracted a lot of attention in the probability and mathematical physics community in the last decade due to its central role in random planar geometry \cite{DMS2014, DS2011} and Liouville conformal field theory \cite{DKRV2016}, and new applications in e.g. random matrix theory \cite{Web2015, LOS2016, BWW2017, NSW2018, CN2019}. Various equivalent constructions of $M_{\gamma}$ have been studied, including the mollification approach which proceeds by considering the weak$^*$ limit of measures
\begin{align*}
M_{\gamma}(dx) = \lim_{\epsilon \to 0} M_{\gamma, \epsilon}(dx) = \lim_{\epsilon \to 0} e^{\gamma X_{\epsilon}(x) - \frac{\gamma^2}{2} \EE[X_{\epsilon}(x)^2]}dx
\end{align*}

\noindent where $X_{\epsilon}(x) = X \ast \theta_{\epsilon}(x)$ for some suitable mollifier $\theta$ (see \Cref{subsec:cGMC}). It is a standard fact that the random measure arising from such limit procedure does not depend on the choice of the mollification, and that it is non-trivial if and only if $\gamma^2 < 2d$, known as the subcritical regime of GMC. We refer the interested readers to \cite{RV2014} for a survey article on subcritical chaos.

In the critical regime where $\gamma = \sqrt{2d}$, we now understand that a non-trivial measure $\mu_{f}$, known as the critical GMC\footnote{The notation $\mu_f$ emphasises the dependence on $f$ appearing in the covariance kernel \eqref{eq:cov} of our underlying field.}, may be constructed via different renormalisation schemes such as the Seneta-Heyde norming
\begin{align*}
\mu_f(dx) 
&= \lim_{\epsilon \to 0^+} \sqrt{\log \frac{1}{\epsilon}} M_{\sqrt{2d}, \epsilon}(dx),
\end{align*}

\noindent at least when $f$ in \eqref{eq:cov} is sufficiently regular (see \Cref{subsec:Gdec}). The theory of critical GMCs is far less developed, and there have been more research efforts in this direction in recent years motivated not only by the study of Liouville quatum gravity at criticality, but also its connection to extremal process of discrete log-correlated Gaussian fields \cite{BL2014,BL2016,BL2018}.

\subsection{Main result: universal tail profile of critical GMCs} \label{subsec:result}
The goal of the present article is to study the tail profile of $\mu_f$, as part of the programme of understanding finer distributional properties of GMCs. This was initially motivated by a question from discrete Gaussian free field (see \Cref{subsec:DGFF}), as well as the following power law result obtained by the author for the subcritical regime.
\begin{theo}\label{theo:subcritical}
Let $\gamma \in (0, \sqrt{2d})$ and $M_{\gamma}$ be the subcritical GMC associated to the log-correlated field $X(\cdot)$ in \eqref{eq:cov}. Suppose the function $f$ appearing in the covariance kernel can be decomposed as
\begin{align*}
f(x, y) = f_+(x, y) - f_-(x, y)
\end{align*}

\noindent where $f_{\pm}(x, y)$ are covariance kernels of some continuous Gaussian fields on $\overline{D}$. Then for any open set $A \subset D$ and continuous function $g \ge 0$ on $\overline{A}$, we have
\begin{align*}
\PP\left(\int_A g(x) M_{\gamma}(dx) > t\right) \overset{t \to \infty}{\sim} \left(\int_A e^{\frac{2d}{\gamma}(Q-\gamma)f(v,v)}g(v)^{\frac{2d}{\gamma^2}}dv\right) \frac{\frac{2}{\gamma}(Q-\gamma)}{\frac{2}{\gamma}(Q-\gamma)+1} \frac{\overline{C}_{\gamma, d}}{t^{\frac{2d}{\gamma^2}}}
\end{align*}

\noindent where $Q = \frac{\gamma}{2}+\frac{d}{\gamma}$. The constant $\overline{C}_{\gamma, d} \in (0, \infty)$, depending on $\gamma$ and $d$ but not on $A$, $f$ or $g$, has a probabilistic representation in all dimensions and explicit formulae when $d \le 2$:
\begin{align*}
\overline{C}_{\gamma, d}
= \begin{dcases}
\frac{(2\pi)^{\frac{2}{\gamma}(Q-\gamma)}}{\frac{\gamma}{2}(Q - \gamma) \Gamma\left(\frac{\gamma}{2}(Q - \gamma)\right)^{\frac{2}{\gamma^2}}}, & d = 1, \\
- \frac{\left(\pi \Gamma(\frac{\gamma^2}{4})/\Gamma(1-\frac{\gamma^2}{4})\right)^{\frac{2}{\gamma}(Q - \gamma)}}{\frac{2}{\gamma}(Q- \gamma)} \frac{\Gamma(-\frac{\gamma}{2}(Q-\gamma))}{\Gamma(\frac{\gamma}{2}(Q-\gamma))\Gamma(\frac{2}{\gamma}(Q-\gamma))}, & d= 2.
\end{dcases}
\end{align*}
\end{theo}

Given the subcritical result, one may make various conjectures regarding the tail probability of $\int_A g(x) \mu_f(dx)$. For instance one may expect that a power law with exponent $1$ will hold in the critical setting, which is consistent with the criterion of existence of moments of critical GMCs that has been known since the work of Duplantier--Rhodes--Sheffield--Vargas \cite{DRSV2014a, DRSV2014b}:
\begin{align*}
\EE \left[ \mu_f(A)^q\right] < \infty \qquad \forall q < 1
\end{align*}

\noindent for any non-empty open set $A \subset D$. An even more interesting conjecture concerns the leading order coefficient: when $\gamma = \sqrt{2d}$ one can verify that
\begin{align*}
\int_A e^{\frac{2d}{\gamma}(Q-\gamma)f(v,v)}g(v)^{\frac{2d}{\gamma^2}}dv = \int_A g(v)dv,
\end{align*}

\noindent which suggests the possibility of a new universality phenomenon that the first-order asymptotics
\begin{align*}
\PP\left(\int_A g(x) \mu_{f}(dx) > t \right) \overset{t \to \infty}{\sim} \frac{\overline{C}_d \int_A g(v)dv}{t}
\end{align*}

\noindent is completely independent of the function $f$ that governs the covariance of our underlying field $X(\cdot)$. Our main result confirms this behaviour.

\begin{theo}\label{theo:main}
Let $\mu_f$ be the critical GMC associated to the log-correlated field $X(\cdot)$ in \eqref{eq:cov}. Suppose the function $f$ appearing in the covariance kernel can be decomposed as
\begin{align}\label{eq:f_dec}
f(x, y) = f_+(x, y) - f_-(x, y)
\end{align}

\noindent where $f_\pm$ are covariance kernels of some continuous Gaussian fields on $\overline{D}$. Then for any Jordan measurable open\footnote{Can be replaced by Borel sets since any the boundary of a Jordan measurable set, which has Lebesgue measure $0$, does not affect the mass with respect to the critical GMC.}sets $A \subset D$ and continuous functions $g \ge 0$ on $\overline{A}$, 
\begin{align}\label{eq:main_result}
\PP\left(\int_A g(x)\mu_{f}(dx) > t\right) \overset{t \to \infty}{=} \frac{\int_A g(v) dv}{\sqrt{\pi d} \cdot t} + o(t^{-1}).
\end{align}
\end{theo}

\begin{rem}
We note that the result in the critical case may be obtained from a heuristic computation based on the subcritical result when we have a closed-form expression for $\overline{C}_{\gamma, d}$. Indeed, using the conjecture\footnote{The claim is proved for the $2$-dimensional Gaussian free field with Dirichlet boundary condition in \cite{APS2018}, and subsequently extended to other log-correlated fields in $d=2$ with the decomposition condition \eqref{eq:f_dec} in \cite{JSW2018}.} that the derivative of subcritical GMCs at $\gamma = \sqrt{2d}^-$ is equivalent to the corresponding critical GMC defined via Seneta-Heyde norming up to a multiplicative factor:
\begin{align}\label{eq:sub_der}
\frac{M_{\gamma}(dx)}{\sqrt{2d} - \gamma} \overset{\gamma \to \sqrt{2d}^-}{\to} \sqrt{2 \pi}\mu_f(dx),
\end{align}

\noindent we would have expected that
\begin{align*}
\PP\left(\int_A \mu_f(dx) > t \right)
& \overset{\gamma \to \sqrt{2d}^-}{\approx} \PP\left(\int_A g(x) M_{\gamma}(dx) > \sqrt{2\pi} (\sqrt{2d}- \gamma) t \right)\\
& \overset{t \to \infty}{\sim} \left(\int_A e^{\frac{2d}{\gamma}(Q-\gamma)f(v,v)}g(v)^{\frac{2d}{\gamma^2}}dv\right) \frac{\frac{2}{\gamma}(Q-\gamma)}{\frac{2}{\gamma}(Q-\gamma)+1} \frac{\overline{C}_{\gamma, d}}{[\sqrt{2\pi}(\sqrt{2d}-\gamma)t]^{\frac{2d}{\gamma^2}}}\\
& \overset{\gamma \to \sqrt{2d}^-}{\to}\bigg(\lim_{\gamma \to \sqrt{2d}^-} \overline{C}_{\gamma, d} \bigg) \frac{\int_A g(v) dv}{\sqrt{\pi d}t}
\end{align*}

\noindent and one could verify using the formulae in \Cref{theo:subcritical} that $\lim_{\gamma \to \sqrt{2d}^-} \overline{C}_{\gamma, d} = 1$ when $d \le 2$. The calculation above is not entirely rigorous, as we have interchanged the (conjectured) limit \eqref{eq:sub_der} and the asymptotics as $t \to \infty$, but it suggests that the approach to critical GMCs from the perspective of derivatives of subcritical GMCs could be a promising one.

\end{rem}

\paragraph{On Jordan measurability.} For technical reasons, the statement of \Cref{theo:main} assumes that $A \subset D$ is Jordan measurable, or equivalently that the boundary $\partial A$ of the set $A$ has zero Lebesgue measure. We discuss the issues in \Cref{app:jordan} and also explain a direct approach in low dimension that may have the possibility of allowing one to circumvent the issues.

\paragraph{On the relevance of kernel decomposition.}
The condition \eqref{eq:f_dec}, which was also present in the subcritical result,  is a technical but very weak assumption that is satisfied by all the important examples like different variants of Gaussian free field in $d=2$ and (regular) $*$-scale invariant fields in any dimension, and can be checked by a tractable criterion regarding the regularity of the function $f$, see \Cref{subsec:Gdec}. 

Similar to the subcritical case, we conjecture that the result holds without this extra condition. Indeed if we assume that $\PP(\int_A g(x) \mu_f(dx) > t)$ satisfies an asymptotic power law profile as $t \to \infty$, then our proof in \Cref{subsec:eval} immediately implies that the power law exponent must equal $1$ and the tail probability has to be of the form \eqref{eq:main_result}. At the moment, however, even the construction of critical GMCs for general $f$ remains to be an open problem.

\subsection{Previous work and our approach}
Despite being of fundamental importance, the tail probability of critical GMC has not been investigated in the literature except\footnote{\cite{BL2014} claims to have obtained the tail asymptotics when $X$ is the $2$-dimensional GFF in their Corollary 2.10 but their argument is based on a result that does not hold, see our \Cref{rem:faketau} in the next subsection.} in \cite{BKNSW2015} where the authors there studied the $L$-exact fields $\EE[Y_L(x) Y_L(y)] = -\log |x-y| + L$ on $[0,1]^d$ for $d \le 2$ and showed that the associated critical GMCs $\mu_{L}(d\cdot)$ satisfy
\begin{align*}
\PP\left(\mu_L([0,1]^d) > t \right) \overset{t \to \infty}{\sim} \frac{C(d)}{t}.
\end{align*}

\noindent Their derivation was based on the exact scale invariance of the underlying field and a stochastic fixed point equation, leading to a probabilistic representation for the non-explicit coefficient $C(d)$. It was unclear how their techniques, based on the application of Goldie's implicit renewal theorem, could be directly adapted to deal with general fields \eqref{eq:cov}, density functions $g$ or domain $D$.

Our strategy here is inspired by our previous work on subcritical GMCs \cite{Won2019}, using the fact that $\mu_f$ is localised (since it is not absolutely continuous with respect to the Lebesgue measure), and the heuristic that when $\mu_{f, g}(A):= \int_A g(x)\mu_{f}(dx)$ is large, then most of the mass comes from the neighbourhood of some $\sqrt{2d}$-thick point of the underlying field $X$. In order to derive the asymptotics for the tail probability, we make use of a Tauberian argument (different from the one in \cite{Won2019}) to reformulate our problem in terms of estimates for the Laplace transform
\begin{align*}
\EE\left[1 - e^{-\lambda \mu_{f, g}(A)^2}\right] \qquad \text{as } \lambda \to 0^+,
\end{align*}

\noindent and apply Gaussian comparison at some point in our proof, but the similarity of the analysis in the critical case to that in the subcritical case ends here. Indeed, if we pursued the approach in \cite{Won2019}, we would have to commence with the localisation trick
\begin{align}\label{eq:local_fail}
\EE\left[1 - e^{-\lambda \mu_{f, g}(A)^2}\right]
= \lim_{\epsilon \to 0^+} \int_A \left(\log \frac{1}{\epsilon}\right)^{\frac{1}{2}} \EE\left[\frac{1}{\mu_{f, g, \epsilon}(v, A)}\left(1 - e^{-\lambda \mu_{f, g, \epsilon}(v, A)^2}\right)\right] dv
\end{align}

\noindent where 
\begin{align*}
\mu_{f, g, \epsilon}(v, A) := \int_A \frac{e^{2d f(x, v)} g(x)\mu_{f, \epsilon}(dx)}{\left(|x-v| \vee \epsilon\right)^{2d}},
\qquad \mu_{f, \epsilon}(dx) = \left(\log \frac{1}{\epsilon}\right)^{\frac{1}{2}} e^{\sqrt{2d}X_{\epsilon}(x) - d \EE[X_\epsilon(x)^2]}dx.
\end{align*}

\noindent More precisely, we would need a very explicit expression for the $\epsilon$-limit \eqref{eq:local_fail} before we could carry out further asymptotic analysis in $\lambda$ like those in \cite[Section 3.2]{Won2019}. The evaluation of the limit \eqref{eq:local_fail}, however, is already very involved for our reference log-correlated field (\Cref{subsec:eval}) that has a very special decomposition, and it is not even clear how it can be done in the general case. The use of Goldie's implicit renewal theorem in \cite[Section 3.1]{Won2019} is also completely irrelevant here because the critical GMC associated with a singular density is simply the wrong object to study. In short, the dominated convergence-based argumentation in the subcritical case is doomed to fail here.

To circumvent these issues we pursue a different approach based on a new splitting lemma, which says if $A_{\pm}$ are two disjoint sets, then
\begin{align*}
\PP(\mu_{f, g}(A_+ \cup A_-) > t)
\overset{t \to \infty}{\sim} \PP(\mu_{f, g}(A_+) > t) + \PP(\mu_{f, g}(A_-) > t)
\end{align*}

\noindent or in terms of Laplace transform,
\begin{align*}
\EE\left[1-e^{-\lambda \mu_{f, g}(A_+ \cup A_-)^2}\right] \overset{\lambda \to \infty}{\sim}
\EE\left[1-e^{-\lambda \mu_{f, g}(A_+)^2}\right] + \EE\left[1 - e^{-\lambda \mu_{f, g}(A_-)^2}\right].
\end{align*}

\noindent This is consistent with the heuristics explained earlier as well as the formula for the leading order asymptotics \eqref{eq:main_result}. Our proof then proceeds as follows:
\begin{itemize}
\item[(i)] We extend the result in \cite{BKNSW2015} for exactly scale-invariant fields, showing that
\begin{align*}
\PP \left(\int_A g(x) \mu_{L}(dx) > t \right) \overset{t \to \infty}{\sim} \frac{\overline{C}_d \int_A g(x)dx}{t}
\end{align*}

for continuous functions $g$ in general dimension $d$.

\item[(ii)] We split our set $A$ into sufficiently small pieces $A_i$ such that the functions $f$ (in \eqref{eq:cov}) and $g$ do not fluctuate a lot on each of them. Gaussian interpolation is then performed directly on each of $\EE\left[1 - e^{-\lambda \mu_{f, g}(A_i)^2}\right]$, resulting in
\begin{align*}
\PP \left(\mu_{f, g}(A) > t \right) \overset{t \to \infty}{\sim} \frac{\overline{C}_d \int_A g(x)dx}{t}.
\end{align*}

\item[(iii)] To evaluate the constant $\overline{C}_d$, we study the asymptotics of
\begin{align}\label{eq:localintro}
\EE\left[\mu_{f, g}(A) e^{-\lambda \mu_{f, g}(A)}\right]
= \lim_{\epsilon \to 0^+} \int_A \left(\log \frac{1}{\epsilon}\right)^{\frac{1}{2}} \EE\left[e^{-\lambda \mu_{f, g, \epsilon}(v, A)}\right] g(v)dv
\end{align}

as $\lambda \to 0^+$ for some convenient reference log-correlated field $X(\cdot)$. This will rely on techniques from recent studies of fusion asymptotics \cite{BW2018}.
\end{itemize}

\subsection{Extremal process of discrete log-correlated fields}\label{subsec:DGFF}
One of our motivations for the present work came from the discrete probability, where there have been a lot of interests in the last decades in understanding the geometry of discrete log-correlated fields. The most studied example is the discrete Gaussian free field (DGFF), which is a centred Gaussian process $\{h_V(x): x \in V\}$ indexed by the finite subset $V$ of vertices of an infinite graph, with covariance $\EE[h_V(x) h_V(y)]$ proportional to the discrete Green's function, i.e. the expected number of visits to y of a simple random walk on the graph starting from $x$ before exiting $V$. 

Let $D \subset \RR^2$ be a bounded domain and $V = V_N := D \cap \frac{1}{N} \ZZ^2$. By choosing the correct proportionality constant, the DGFF $ h_N(\cdot) := h_{V_N}(\cdot) $ may be defined so that
\begin{align*}
\EE\left[h_N(x) h_N(y)\right] = -\log \left(|x-y| \vee \frac{1}{N}\right) + \mathcal{O}(1).
\end{align*}

\noindent Under this normalisation, it is well-known since the work of Bramson and Zeitouni \cite{BZ2011} that $m_N = 2 \log N - \frac{3}{4} \log \log N$ captures the rate of growth of $\max_{x \in V_N} h_N(x)$ as $N \to \infty$. This led to a sequence of work \cite{BL2014, BL2016, BL2018} by Biskup and Louidor who studied the extremal process of the DGFF, i.e. the scaling limit of
\begin{align*}
\eta_N(dx, dh) = \sum_{x \in V_N} \delta_{v}(dx) \otimes \delta_{h_N(v) - m_N}(dh) 1_{\{h_N(v) = \max_{|y-v| < r_N} h_N(y) \}}
\end{align*}

\noindent as $N \to \infty$. Under suitable condition on the domain $D$ and $N^{-1} \ll r_N \ll 1$, they showed that
\begin{align*}
\eta_N(dx, dh) \xrightarrow[N \to \infty]{d} \mathrm{PPP}\left(\mu(dx) \otimes e^{-2h}dh\right)
\end{align*}

\noindent where $\mu$ is a random measure characterised by a collection of five axioms (\cite[Theorem 2.8]{BL2014}). It was long conjectured that $\mu$ should be the critical Liouville quantum gravity measure $\mu_{\mathrm{LQG}}$ up to a deterministic factor, i.e. the critical GMC measure $g(x) \mu_f(dx)$ associated to
\begin{itemize}
\item the Gaussian free field on $D$, i.e. $\EE[X(x) X(y)] = G_D(x, y)$ where $G_D(x, y)$ is the Green's function with Dirichlet boundary condition, and
\item $g(x) = R(x; D)^2$ is the square of the conformal radius.
\end{itemize}

The main difficulty in resolving the conjecture lay in the verification of the fifth axiom, i.e. whether
\begin{align}\label{eq:5thaxiom}
\forall A \subset D, \qquad \lim_{\lambda \to 0^+} \frac{ \EE \left[\mu(A) e^{-\lambda \mu(A)}\right]}{-\log \lambda} = c \int_A R(x; D)^2 dx
\end{align}

\noindent was satisfied by the choice of $\mu = \mu_{\mathrm{LQG}}$, and one should not be surprised that the above claim would have followed from the tail asymptotics of $\mu_{\mathrm{LQG}}(A)$.

At the time the project started, the verification of \eqref{eq:5thaxiom} for $\mu = \mu_{\mathrm{LQG}}$ had been left as an open problem for four years, but it was finally resolved when there was a revision of the preprint \cite{BL2014} where the authors there evaluated the limit \eqref{eq:5thaxiom} based on a localisation trick similar to the one we employ for the evaluation of our coefficient $\overline{C}_d$. In contrast to the analysis in \cite[Section 8]{BL2014}, our analysis in \Cref{subsec:eval} is based on a fully continuum approach inspired by recent analysis of fusion limits in Liouville conformal field theory \cite{BW2018} with the advantage of being able to determine the limit
\begin{align*}
\lim_{\epsilon \to 0^+} \left(\log \frac{1}{\epsilon}\right)^{\frac{1}{2}} \EE\left[e^{-\lambda \mu_{f, \epsilon}(v, A)}\right] 
\end{align*}

\noindent for each $\lambda > 0$, at least for our choice of reference field $X$, which could be of independent interest. In some special cases (e.g. when $X$ is the exact field in $d \le 2$) we can use the same technique to study the localisation trick
\begin{align*}
\EE\left[ e^{-\lambda / \mu_{f, g}(A)}\right] = \int_A \lim_{\epsilon \to 0^+} \left(\log \frac{1}{\epsilon}\right)^{1/2} \EE\left[\frac{1}{\mu_{f, g, \epsilon}(v, A)} e^{-\lambda / \mu_{f, g, \epsilon}(v, A)}\right] dv
\end{align*}

\noindent and the ability to evaluate the limit of the integrand on the RHS for each $\lambda > 0$ means that not only can we obtain the first order asymptotics of $\PP\left(\mu_{f, g}(A) > t \right)$, there is also a possibility to obtain some bound on the lower order terms via Tauberian remainder theorem.

\begin{rem}\label{rem:faketau}
Based on the asymptotics \eqref{eq:5thaxiom} for $\mu = \mu_{\mathrm{LQG}}$, \cite[Corollary 2.10]{BL2014} claimed that 
\begin{align*}
\PP\left(\mu_{\mathrm{LQG}}(A) > t \right) \overset{t \to \infty}{\sim} \frac{c \int_A R(x; D)^2 dx}{t}
\end{align*}

\noindent could be readily deduced from the use of Tauberian theorems. We note that this claim is false as it would have relied on a proposition that
\begin{align} \label{eq:falsetau}
\PP(U > t) \overset{t \to \infty}{\sim} \frac{C}{t} \qquad \Leftrightarrow \qquad \EE \left[U e^{-\lambda U}\right] \overset{\lambda \to 0^+}{\sim} -C \log \lambda
\end{align}

\noindent for non-negative random variables $U$. While the forward implication is always true, unfortunately the backward direction of \eqref{eq:falsetau} does not always hold, e.g. consider the counterexample $\PP(U > t) = (1+0.0001 \sin (\log t)) / t$ for $t \ge 1$. Indeed the same counterexample shows that the asymptotics for $\EE[Ue^{-\lambda U}]$ is so weak that it does not even imply $\PP(U  > t) \sim L(t) / t$ for some slowly-varying function\footnote{A function $L(\cdot)$ is said to be slowly-varying at infinity if $\lim_{t \to \infty} L(xt) / L(t) = 1$ for any $x > 0$.} $L(\cdot)$ at infinity.

This highlights the fact that it is a more challenging task to establish the tail universality of $\mu_{f, g}(A)$. On one hand, it is easy to show that
\begin{align*}
\lim_{\lambda \to 0^+} \frac{\EE\left[\mu_{f, g}(A) e^{-\lambda \mu_{f, g}(A)}\right]}{-\log \lambda}
\end{align*}

\noindent does not depend on $f$ because the localisation trick allows us to rewrite $\EE\left[\mu_{f, g}(A) e^{-\lambda \mu_{f, g}(A)}\right]$ in terms of $\EE\left[e^{-\lambda \mu_{f, g, \epsilon}(v, A)}\right]$ as in \eqref{eq:localintro}, and the latter expectation involves a convex evaluation of a GMC mass which means Gaussian comparisons can be performed without incurring huge error (in the sense that the difference between the upper and lower bounds are negligible when divided by $-\log \lambda$ in the limit). On the other hand, any Laplace transform estimate that is equivalent to the tail probability asymptotics by Tauberian argument is going to involve evaluation at a GMC mass of some function that is not convex or concave (see e.g. the integrand on the RHS of \eqref{eq:local_fail}), making it impossible to apply Kahane's convexity inequality \eqref{eq:Gcomp} to reduce our problem to special log-correlated fields in the first step.
\end{rem}

\subsection{Outline of the paper}
The remainder of the article is organised as follows.

In \Cref{sec:prelim} we collect a few results that will be used in later parts of the paper. This includes basic facts about Gaussian processes, decomposition of log-correlated fields and construction of Gaussian multiplicative chaos, Tauberian theorems as well as properties of $3$-dimensional Bessel processes.

\Cref{sec:proof} is devoted to the proof of our main theorem and a more elaborate outline of our three-step approach may be found in that section. Various technical results, namely the cross moment estimates of GMCs (\Cref{lem:cross}), the construction of reference GMCs \eqref{eq:def_Sd} as well as the evaluation of localisation limit (\Cref{lem:fusion}), are proved in \Cref{app:cross,app:ref,app:fusion}. We also discuss the technical assumption of Jordan measurability in \Cref{app:jordan}.

\paragraph{Acknowledgement} 
The present work is supported by ERC Advanced Grant 740900 (LogCorRM).

\section{Preliminaries}\label{sec:prelim}
\subsection{Gaussian processes}\label{subsec:GP}
We begin with the following lemma, which is a less precise version of concentration of suprema of Gaussian processes (see e.g. \cite[Section 2.1]{Won2019}).
\begin{lem} \label{lem:ctsGP}
Let $G(\cdot)$ be a continuous Gaussian field on some compact domain $K \subset \RR^d$, then there exists some $c > 0$ such that
\begin{align}\label{eq:fast_tail}
\PP\left( \sup_{x \in K} |G(x)| > t \right) \le \frac{1}{c} e^{-c t^2}, \qquad \forall t \ge 0.
\end{align}
\end{lem}

The following enhanced continuity criterion for Gaussian processes is due to \cite{ASVY2014}.
\begin{lem}\label{lem:contGP}
Let $(G(x))_{x \in [0,T]}$ be a centred Gaussian process. Then it is $\alpha$-H\"older continuous for any $\alpha < H$, i.e.
\begin{align}\label{eq:Holder}
|G(x) - G(y)| \le C_\epsilon |x-y|^{H-\epsilon}, \qquad \forall \epsilon > 0
\end{align}

\noindent if and only if there exists $c_\epsilon > 0$ such that
\begin{align*}
\EE \left[ (G(x) - G(y))^2 \right]^{1/2} \le c_\epsilon |x-y|^{H-\epsilon}, \qquad \forall \epsilon > 0.
\end{align*}

\noindent Moreover, the random variable $C_\epsilon$ in \eqref{eq:Holder} satisfies $\EE\left[\exp(a C_\epsilon^k) \right] < \infty$ for any $a \in \RR$ if $k < 2$, or $a > 0$ small enough if $k = 2$.

\end{lem}

\subsection{Decomposition of Gaussian fields}\label{subsec:Gdec}
Let $f(x, y)$ be a symmetric function on $D \times D$ for some domain $D \subset \RR^d$. We say $f$ is in  the local Sobolev space $H_{\mathrm{loc}}^s(D \times D)$ of index $s > 0$ if $\kappa f$ is in $H^s(D \times D)$ for any $\kappa \in C_c^\infty(D \times D)$, i.e.
\begin{align*}
\int_{\RR^d} (1+|\xi|^2)^s|\widehat{(\kappa f)}(\xi)|^2 d\xi < \infty
\end{align*}

\noindent where $\widehat{(\kappa f)}$ is the Fourier transform of $\kappa f$. We now mention a convenient criterion\footnote{We implicitly assume that $f$ is continuous on $\overline{D} \times \overline{D}$, which is necessary if it is the difference of covariance kernels of two continuous Gaussian fields.} for checking whether our log-correlated Gaussian field \eqref{eq:cov} satisfies the decomposition condition \eqref{eq:f_dec}. 
\begin{lem}[{cf. \cite[Lemma 3.2]{JSW2018}}]\label{lem:holder_dec}
If $f \in H_{\mathrm{loc}}^s(D \times D)$ for some $s > d$, then there exist two centred, H\"older-continuous Gaussian processes $G_\pm$ on $\RR^d$ such that
\begin{align}
\EE[G_+(x) G_+(y)] - \EE[G_-(x) G_-(y)] = f(x, y), \qquad \forall x, y \in D'
\end{align}

\noindent for any bounded open set $D'$ such that $\overline{D'} \subset D.$
\end{lem}

An interesting implication of \Cref{lem:holder_dec} (using a further decomposition result of \cite{JSW2018}) is that the logarithmic kernel, when restricted to sufficiently small Euclidean ball, is positive definite. This may be seen as a trivial special case of \cite[Theorem B]{JSW2018} and has been known since \cite{RV2010} by a different, spherical averaging argument.
\begin{lem}\label{lem:exact}
For each $L \in \RR^d$, there exists $r_d(L) > 0$ such that the kernel
\begin{align} \label{eq:L_exact}
K_L(x, y) = -\log|x-y| + L
\end{align}

\noindent is positive semi-definite on $B(0, r_d(L))$. In particular, there exists a Gaussian field $Y_L$ on $B(0, r_d(L))$ with covariance kernel given by $K_L$.
\end{lem}

Following \cite{Won2019}, we shall call the function $K_L$ in \eqref{eq:L_exact} the $L$-exact kernel, and when $L = 0$ we simply call $K_0$ the exact kernel and write $r_d = r_d(0)$. The associated Gaussian field $Y_L$ will be called the $L$-exact field (or the exact field when $L = 0$), and without loss of generality $r_d(L)$ is chosen such that it is a nondecreasing function in $L$. As we shall see later, the $L$-exact fields will play a pivotal role in local approximations in \Cref{subsec:universal}, as well as in the construction of our reference log-correlated field in dimension $d \ge 2$ on which we depend for the identification of the proportionality constant $\overline{C}_d$ in \Cref{subsec:eval} and \Cref{app:fusion}.

\subsection{Critical Gaussian multiplicative chaos}\label{subsec:cGMC}
There are two equivalent constructions of critical Gaussian multiplicative chaos, namely the derivative martingale approach and Seneta-Heyde renormalisation, which were first studied in \cite{DRSV2014a} and \cite{DRSV2014b} respectively for the special class of $*$-scale invariant kernels. Combined with the mollification method, they have been recently extended to treat more general log-correlated fields \cite{JS2017,Pow2017,JSW2018}. 

Without loss of generality we shall focus on the Seneta-Heyde renormalisation, which defines the critical GMC via the limit
\begin{align*}
\mu_{f}(dx) := \lim_{\epsilon \to 0^+} \left( \log \frac{1}{\epsilon}\right)^{\frac{1}{2}} e^{\sqrt{2d} X_{\epsilon}(x) - d \EE[X_\epsilon(x)^2]} dx
\end{align*}

\noindent where $X_{\epsilon}(x) = X \ast \theta_\epsilon(x)$ for suitable mollifiers $\theta_\epsilon(x) := \epsilon^{-d} \theta(x / \epsilon)$. For simplicity and definiteness we shall consider $\theta \in C_c^\infty(\RR^d)$ such that $\theta \ge 0$ and $\int \theta = 1$, but more general condition is available (see e.g. \cite{Pow2017}).
\begin{lem}\label{lem:criticalGMC}
Let $X$ be the log-correlated Gaussian field in \eqref{eq:cov} satisfying the decomposition condition \eqref{eq:f_dec}. Then the sequence of measures
\begin{align*}
\mu_{f, \epsilon}(dx) := \left( \log \frac{1}{\epsilon}\right)^{\frac{1}{2}} e^{\sqrt{2d} X_{\epsilon}(x) - d \EE[X_\epsilon(x)^2]} dx
\end{align*}

\noindent converges in probability as $\epsilon \to 0^+$ to some locally finite random Borel measure $\mu_f(dx)$ on $D$ in the weak$^*$ topology. Moreover, the limit $\mu_f(dx)$, formally written as $\mu_f(dx) = e^{\sqrt{2d}X(x) - d \EE[X(x)^2]}dx$, is independent of the choice of mollifiers $\theta$.
\end{lem}

Let us collect Kahane's interpolation formula, which is a very useful tool in the theory of multiplicative chaos. We first state the result for continuous Gaussian fields.
\begin{lem}[\cite{Kah1985}] \label{lem:Kahane}
Let $\rho$ be a Radon measure on $D$, $X(\cdot)$ and $Y(\cdot)$ be two continuous centred Gaussian fields, and $F: \RR_+ \to \RR$ be some smooth function with at most polynomial growth at infinity. For $t \in [0,1]$, define $Z_t(x) = \sqrt{t}X(x) + \sqrt{1-t}Y(x)$ and
\begin{align}
\varphi(t) := \EE \left[ F(W_t)\right], \qquad
W_t := \int_D e^{Z_t(x) - \frac{1}{2}\EE[Z_t(x)^2]} \rho(dx).
\end{align}

\noindent Then the derivative of $\varphi$ is given by
\begin{equation}\label{eq:Kahane_int}
\begin{split}
\varphi'(t) & = \frac{1}{2} \int_D \int_D \left(\EE[X(x) X(y)] - \EE[Y(x) Y(y)]\right) \\
& \qquad \qquad \times \EE \left[e^{Z_t(x) + Z_t(y) - \frac{1}{2}\EE[Z_t(x)^2] - \frac{1}{2}\EE[Z_t(y)^2]} F''(W_t) \right] \rho(dx) \rho(dy).
\end{split}
\end{equation}

\noindent In particular, if
\begin{align*}
\EE[X(x) X(y)] \le \EE[Y(x) Y(y)] \qquad \forall x, y \in D,
\end{align*}

\noindent then for any convex $F: \RR_+ \to \RR$ with at most polynomial growth at infinity,
\begin{align}\label{eq:Gcomp}
\EE \left[F\left(\int_D e^{X(x) - \frac{1}{2} \EE[X(x)^2]}\rho(dx)\right)\right]
\le
\EE \left[F\left(\int_D e^{Y(x) - \frac{1}{2} \EE[Y(x)^2]}\rho(dx)\right)\right].
\end{align}

\noindent The inequality is reversed if $F$ is concave instead.
\end{lem}

The comparison principle \eqref{eq:Gcomp} may easily be used in the study of log-correlated fields if we apply it to mollified fields $X_{\epsilon}$ and $Y_{\epsilon}$, a standard argument in the GMC literature which we shall take for granted. For the interpolation principle \eqref{eq:Kahane_int}, we only need the following exponential version of \cite[Corollary 2.7]{Won2019} with $F(x) = 1-e^{-\lambda x^2}$ (see proof of \Cref{lem:universal}), which may be extended to log-correlated fields by taking the limit $\epsilon \to 0^+$.

\begin{cor}\label{cor:interpolate}
Under the same assumptions and notations in \Cref{lem:Kahane}, if there exists some $C>0$ such that
\begin{align*}
\left|\EE[X(x) X(y)] - \EE[Y(x) Y(y)]\right| \le C \qquad \forall x, y \in D,
\end{align*}

\noindent then
\begin{align*}
|\varphi'(t)| \le \frac{C}{2} \EE \left[ (W_t)^2 |F''(W_t)|\right].
\end{align*}

\noindent In particular, if $F$ is also non-negative then
\begin{align*}
\exp \left(-\frac{C}{2} \int_0^1 \frac{\EE \left[ (W_t)^2 |F''(W_t)|\right]}{\EE \left[F(W_t)\right]} dt\right) \le \frac{\varphi(1)}{\varphi(0)} \le \exp \left(\frac{C}{2} \int_0^1 \frac{\EE \left[ (W_t)^2 |F''(W_t)|\right]}{\EE \left[F(W_t)\right]} dt\right).
\end{align*}
\end{cor}

We now compile a list of results regarding the moments of critical GMCs which has been known since \cite{DRSV2014a, DRSV2014b} (or \cite{AS2013,BKNSW2014} for analogous results for branching random walk and cascades respectively).
\begin{lem}\label{lem:GMC_moment}
Let $q < 1$.

\begin{itemize}
\item[(i)] If $A \subset D$ is a non-empty open set, then $\sup_{\epsilon \in (0, 1)} \EE \left[ \mu_{f, \epsilon}(A)^q\right] < \infty$. In particular,
\begin{align*}
\EE \left[ \mu_{f}(A)^q\right] \le \liminf_{\epsilon \to 0^+} \EE \left[ \mu_{f, \epsilon}(A)^q\right] < \infty.
\end{align*}

\item[(ii)] Let $x \in D$. Then there exists some $C\in (0, \infty)$ possibly depending on $L, q, A$ but not on $r \in (0, 1)$ such that
\begin{align*}
\EE \left[ \mu_{f}(B(x, r))^q\right] 
\le \liminf_{\epsilon \to 0^+} \EE \left[ \mu_{f, \epsilon}(B(x, r))^q\right]
\le \limsup_{\epsilon \to 0^+} \EE \left[ \mu_{f, \epsilon}(B(x, r))^q\right]
\le C r^{2dq - dq^2}
\end{align*}

\noindent for all log-correlated Gaussian fields \eqref{eq:cov} satisfying $||f||_\infty \le L$.
\end{itemize}
\end{lem}

Let us also collect the following estimates regarding the regularised field $X_\epsilon$.
\begin{lem}[{\cite[Lemma 3.5]{Ber2017}}] Let $X(\cdot)$ be a log-correlated Gaussian field with covariance \eqref{eq:cov} where $f$ is continuous on $\overline{D} \times \overline{D}$. For each $\epsilon > 0$, we have
\begin{align*}
\EE[X_\epsilon(x) X_\epsilon(y)] = - \log \left(|x-y| \vee \epsilon \right) + f_\epsilon(x, y) \qquad \forall x, y \in D
\end{align*}

\noindent where $f_\epsilon(x, y) = O(1)$ as $|x-y| \to 0$. Moreover, $f_\epsilon(x, y)$ converges pointwise to $f(x, y)$ for any $x \ne y$, and if $\delta > 0$ then
\begin{align*}
\EE[X_\epsilon(x) X_\epsilon(y)] = - \log |x-y| + f(x, y) +o(1) \qquad \text{as } \epsilon \to 0^+
\end{align*}

\noindent uniformly in $|x-y| \ge \delta$.
\end{lem}

\subsection{Tauberian theorem and related auxiliary results} \label{subsec:tau}
We now state Karamata's classical Tauberian theorem.
\begin{theo}[{\cite[Theorem XIII.5.3]{Fel1971}}] \label{theo:tau}
Let $\nu(d\cdot)$ be a non-negative measure on $\RR_+$, $F(s):= \int_0^t \nu(ds)$ and suppose
\begin{align*}
\widetilde{F}(\lambda) := \int_0^\infty e^{-\lambda s} \nu(ds)
\end{align*}

\noindent exists for $\lambda > 0$. If $C \ge 0$ and $\rho \in [0, \infty)$, then
\begin{align*}
\widetilde{F}(\lambda) \overset{\lambda \to 0^+}{\sim} C \lambda^{-\rho}
\qquad \Leftrightarrow \qquad
F(t) \overset{t \to \infty}{\sim} \frac{C}{\Gamma(1+\rho)} t^\rho.
\end{align*}

\noindent The equivalence  still holds if we consider the above asymptotics with $\lambda \to \infty$ and $t \to 0^+$ instead.
\end{theo}

\begin{rem}
The case where $C=0$ in \Cref{theo:tau} should be interpreted as
\begin{align*}
\widetilde{F}(\lambda) \overset{\lambda \to 0^+}{=} o( \lambda^{-\rho})
\qquad \Leftrightarrow \qquad
F(t) \overset{t \to \infty}{=} o(t^\rho).
\end{align*}
\end{rem}

We also need the following elementary result, the proof of which is skipped.
\begin{lem} \label{lem:aux}
Let $U, V$ be two independent non-negative random variables. Suppose there exists some $C > 0$ and $q > 0$ such that
\begin{align*}
(i) & \qquad \PP(U > t) \overset{t \to \infty}{\sim} C t^{-q}, \\
(ii) & \qquad \EE[V^p] < \infty \qquad \text{for some } p > q.
\end{align*}

\noindent Then the tail behaviour of $UV$ is given by
\begin{align*}
(iii) & \qquad \PP(UV > t) \overset{t \to \infty}{\sim} C \EE[V^q] t^{-q}. \qquad 
\end{align*}

Also, if condition (i) is replaced by 
\begin{align*}
& \qquad \PP(U > t) \le C t^{-q} \qquad \forall t > 0,\\
 (resp. & \qquad \PP(U > t) \ge Ct^{-q} \qquad \forall t > t_0,)
\end{align*} then the conclusion becomes
\begin{align*}
& \qquad \PP(UV > t) \le C \EE[V^q] t^{-q} \qquad \forall t > 0. \\
(resp. & \qquad \PP(UV > t) \ge C \EE \left[\frac{V^q}{a^q} 1_{\{V \le a\}}\right] t^{-q} \qquad \forall t > t_0 \qquad \text{for any } a > 1.)
\end{align*}

\end{lem}

\subsection{Three-dimensional Bessel processes}
Finally we collect several results regarding three-dimensional Bessel processes (abbreviated as $\mathrm{BES}(3)$-processes), which will be important for the evaluation of the proportionality constant in \Cref{subsec:eval} and \Cref{app:fusion}. The first two results are due to Williams \cite{Wil1974}, see also e.g. \cite[Chapter VII.4]{RY2004}.

The first result relates the time reversal of Brownian motion from the first hitting time to a $\mathrm{BES}(3)$-process evolving until a last hitting time.
\begin{lem}\label{lem:time_rev}
Let $(B_t)_{t \ge 0}$ be a standard Brownian motion. For $x \ge 0$, let $T_x := \inf \{s > 0: B_s = x\}$ be the first hitting time of the Brownian motion. Then
\begin{align*}
(B_{T_x - t})_{t \le T_x} \overset{d}{=} (x - \beta_t^0)_{t \le L_{x}}
\end{align*}

\noindent where $(\beta_t^0)_{t \ge 0}$ is a $\mathrm{BES}(3)$-process starting from $0$ and $L_{x} = \sup \{s > 0: \beta_s^0 = x \}$ is the last hitting time of the Bessel process.
\end{lem}

The second result provides a path decomposition of $\mathrm{BES}(3)$-processes.
\begin{lem}\label{lem:BES_path}
Let $x > 0$ and consider the following independent objects:
\begin{itemize}
\item $(B_t)_{t \ge 0}$ is a standard Brownian motion.
\item $U$ is a $\mathrm{Uniform}[0,1]$ random variable.
\item $(\beta_t^0)_{t \ge 0}$ is a $\mathrm{BES}(3)$-process starting from $0$.
\end{itemize}

\noindent Then the process $(R_t)_{t \ge 0}$ defined by
\begin{align*}
R_t = \begin{cases}
x + B_t & t \le T_{-x(1-U)}, \\
xU + \beta_{t-T_{-x(1-U)}}^0 & t \ge T_{-x(1-U)},
\end{cases}
\end{align*}

\noindent with
\begin{align*}
T_{-x(1-U)} = \inf \{ t > 0: B_t = -x(1-U)\} = \inf \{t > 0: x + B_t = xU \}
\end{align*}

\noindent is a $\mathrm{BES}(3)$-process starting from $x$ (abbreviated as $\mathrm{BES}_x(3)$-process).
\end{lem}

The last result relates Brownian motions to $\mathrm{BES}(3)$-processes via a change of measure.
\begin{lem}\label{lem:RadNik}
Let $(B_t)_{t \ge 0}$ be the standard Brownian motion with its natural filtration $(\Fa_t)_{t \ge 0}$ under the probability measure $\PP$. For any $x > 0$,
\begin{itemize}
\item The process $t \mapsto 1_{\{\max_{s \le t} x - B_s\}} (x-B_s)$ is an $(\Fa_t)_t$-martingale with respect to $\PP$. In particular $\EE\left[1_{\{\max_{s \le t} x - B_s\}} (x-B_s)\right] = x$ for any $t > 0$.
\item The collection of probability measures $(\QQ_t^x)_{t \ge 0}$ defined by the Radon-Nikodym derivative
\begin{align*}
\frac{d\QQ_t^x}{d\PP} ((u_t)_t) = \frac{1}{x} 1_{\{\max_{s \le t} u_s \le x\}} (x - u_t)
\end{align*}

\noindent is compatible in the sense that $\QQ_t^x|_{\Fa_s} = \QQ_s^x$ for any $s < t$. In particular, there exists a probability measure $\QQ^x$ on $\Fa_\infty$ such that $\QQ^x|_{\Fa_t} = \QQ_t^x$ and under which the path
\begin{align*}
t \mapsto x - B_t
\end{align*}

\noindent evolves as a $\mathrm{BES}_x(3)$-process.
\end{itemize}
\end{lem}

\section{Main proofs} \label{sec:proof}
We give an outline of this section, which is devoted to the proof of \Cref{theo:main}.

In \Cref{subsec:partial}, we state a partial result regarding the tail probability of critical GMCs for $L$-exact kernels. The result is incomplete as it only applies to $g(x) \equiv 1$ on $A = (0, a)^d$ and our goal is to understand how the leading order coefficient behaves as we vary $L$ and $a$.

In \Cref{subsec:re}, we reformulate the desired tail asymptotics \eqref{eq:main_result} as a Laplace transform estimate, and explain several reductions that may be achieved and shall be assumed in the rest of our proof.

In \Cref{subsec:split}, we present our simple yet important splitting lemma and extend the partial tail result in \Cref{subsec:partial} to continuous density functions $g(x)$.

In \Cref{subsec:universal}, we explain how the splitting lemma allows for local approximations, and ultimately prove that
\begin{align*}
\PP\left(\mu_{f, g}(A) > t \right) \overset{t \to \infty}{\sim} \frac{\overline{C}_d \int_A g(v) dv}{t}
\end{align*}

\noindent for some constant $\overline{C}_d$ that does not depend on the function $f$ appearing in the covariance structure \eqref{eq:cov}, the test set $A$ or the density function $g$.

Finally, in \Cref{subsec:eval}, we evaluate the mysterious constant $\overline{C}_d$ that appears in the leading order coefficient, by considering suitable reference critical GMCs that allow tractable calculations.

\subsection{A partial tail result} \label{subsec:partial}
We commence with a result concerning exact kernels in arbitrary dimension $d$.
\begin{lem}\label{lem:partial}
Let $Y_L$ be the $L$-exact field on $B(0, r_d(L))$ and $\mu_L(dx)$ the associated critical GMC. Let $a > 0$ be a fixed number such that $[0, a]^d \subset B(0, r_d(L))$. Then there exists some constant $C_{L, a, d} > 0$ such that
\begin{align}\label{eq:partial}
\PP\left(\mu_{L}([0, a]^d) > t\right) \overset{t \to \infty}{\sim} \frac{C_{L, a, d}}{t}.
\end{align}
\end{lem}

\Cref{lem:partial} is essentially due to \cite{BKNSW2015}, where the $d \le 2$ cases were established as Theorem 1 and Theorem 25 there. Quoting the discussion before the Appendices in \cite{BKNSW2015}, the proof of \Cref{lem:partial} for $d \le 2$ may be extended to higher dimensions immediately as long as one has the existence of the corresponding critical chaos (which has now been addressed) and an estimate analogous of \cite[Lemma 29]{BKNSW2015} for $d \ge 3$. For later applications we state and prove this analogous result for general GMCs.
\begin{lem}\label{lem:cross}
Let $B_1, B_2$ be two disjoint subsets of $D$ separated by a hyperplane, i.e. $B_1 \cap B_2 = \emptyset$ and there exists some $a \in \RR^d$ and $c \in \RR$ such that $\partial B_1 \cap \partial B_2 \subset \{x \in \RR^d: \langle a, x \rangle = c \}$. Then for any $h \in [0, \frac{1}{2} + \frac{1}{2\sqrt{d}})$ we have
\begin{align}\label{eq:cross_bound}
\EE\left[\mu_{f, g}(B_1)^h \mu_{f, g}(B_2)^h \right] < \infty.
\end{align}
\end{lem}

The proof of \Cref{lem:cross} is postponed to \Cref{app:cross}, and we refer the readers to \cite{BKNSW2015} for the arguments leading to a proof of \Cref{lem:partial}.

Going back to the statement of \Cref{lem:partial}, we note that the method in \cite{BKNSW2015} provides a probabilistic representation for the constant $C_{L, a,d}$ in \eqref{eq:partial} (which we do not need here) but its value is not known a priori in any dimension. It is easy to show, however, that
\begin{cor}\label{cor:pscale}
There exists some $\overline{C}_d > 0$ such that $C_{L, a, d} = a^d \overline{C}_d$.
\end{cor}

\begin{proof}
Suppose $a < 1$ and $[0,1]^d \subset B(0, r_d(L))$. We can always write
\begin{align*}
\EE[Y_L(x) Y_L(y)] = -\log \left| \frac{x-y}{a}\right| + L - \log a \qquad \forall x, y \in [0,a]^d,
\end{align*}

\noindent i.e. the field $(\widetilde{Y}_L(u))_{u \in [0,1]^d} := (Y_L(au))_{u \in [0,1]^d}$ has the same law as $(Y_L(u) + N_a)_{u \in [0,1]^d}$ where $N_a$ is an independent $N(0, -\log a)$ random variable, and hence
\begin{align*}
\int_{[0,a]^d} \mu_{L}(dx) \overset{d}{=} a^d e^{\sqrt{2d} N_a - d \EE[N_a^2]} \int_{[0,1]^d} \mu_{L}(dx).
\end{align*}

\noindent By \Cref{lem:partial} and \Cref{lem:aux},
\begin{align*}
\PP\left(\mu_L([0,a]^d) > t\right)
& = \PP\left( a^d e^{\sqrt{2d} N_a - d \EE[N_a^2]} \mu_L([0,1]^d) > t\right)\\
& \sim \frac{C_{L, 1, d} \EE\left[a^d e^{\sqrt{2d} N_a - d \EE[N_a^2]}\right]}{t}
= \frac{C_{L, 1,d} a^d}{t}
\end{align*}

\noindent which shows that $\overline{C}_{L, a, d} \propto a^d$. Similarly one can show that the proportionality constant does not depend on $L$ and we are done.
\end{proof}

\subsection{A reformulation and some reductions} \label{subsec:re}
Similar to the subcritical story, it is very useful to reformulate the tail asymptotics with the help of a Tauberian argument.
\begin{lem}\label{lem:reform}
Let $U \ge 0$ be a non-negative random variable. Then 
\begin{align}\label{eq:reform0}
\PP(U > t) \overset{t \to \infty}{\sim} \frac{C}{\sqrt{t}}
\qquad \Leftrightarrow \qquad
\lim_{\lambda \to 0^+} \lambda^{-1/2}\EE \left[1 - e^{-\lambda U}\right] = C \sqrt{\pi}.
\end{align}

\noindent In particular, the tail asymptotics \eqref{eq:main_result} is equivalent to
\begin{align}\label{eq:reform}
\lim_{\lambda \to 0^+} \lambda^{-1/2}\EE \left[1 - e^{-\lambda \mu_{f, g}(A)^2}\right] = d^{-1/2} \int_A g(v)dv.
\end{align}
\end{lem}

\begin{proof}
The forward implication follows from a straightforward computation which is skipped here. Now recall that
\begin{align*}
\EE\left[\frac{1-e^{-\lambda U}}{\lambda} \right] = \int_0^\infty e^{-\lambda u} \PP(U > u) du.
\end{align*}

\noindent By \Cref{theo:tau},  the Laplace transform estimate in \eqref{eq:reform0} is equivalent to
\begin{align}\label{eq:tauapp1}
\int_0^t \PP(U > u) du 
\overset{t \to \infty}{\sim} \frac{C\sqrt{\pi}}{\Gamma(3/2)} \sqrt{t}
= 2C \sqrt{t}
\end{align}

For the purpose of evaluating tail asymptotics, we may assume without loss of generality that $\PP(U > t)$ is continuous in $t \ge 1$. In general $\PP(U>t)$ has at most countably many discontinuities, but for any $t_0 > 1$ we can always find $\epsilon \in (0, 1)$ such that the function is continuous at $t_0 \pm \epsilon$ and
\begin{align*}
\PP(U > t_0 - \epsilon) \ge \PP(U> t_0) \ge \PP(U > t_0 + \epsilon)
\end{align*}

\noindent by monotonicity. Under this reduction, the derivative of LHS of \eqref{eq:tauapp1} with respect to $t$ exists and is equal to $\PP(U > t)$, and the proof can be concluded if we can justify the differentiation of the asymptotics on the RHS of \eqref{eq:tauapp1}

For each $\epsilon > 0$, we have
\begin{align*}
\left|\int_0^t \PP(U > u) du - 2C\sqrt{t} \right| \le \epsilon \sqrt{t}
\end{align*}

\noindent for $t$ sufficiently large. Then
\begin{align*}
h \PP(U > t)
&\ge  \int_t^{t+h} \PP(U > u) du\\
& \ge (2C-\epsilon) \sqrt{t+h} - (2C+\epsilon) \sqrt{t}
\ge \frac{Ch}{\sqrt{t + h}} - 2 \epsilon \sqrt{t+h}.
\end{align*}

\noindent for any $h \in (0, t)$, and by choosing e.g. $h = \sqrt{\epsilon} t^{2/3}$, we see that
\begin{align*}
\liminf_{t \to \infty} \sqrt{t} \PP(U > t) \ge C.
\end{align*}

\noindent The bound in the other direction may also be obtained by considering the integral in the interval $[t-h, t]$, and we arrive at
\begin{align*}
\lim_{t \to \infty} \sqrt{t} \PP(U > t) = C.
\end{align*}
\end{proof}

\paragraph{Some reductions.} We discuss several reductions of our problem which will be taken for granted in our proof.\\

Given the continuity of $g \ge 0$ on $\overline{A}$, we may assume that 
\begin{center}
\textbf{R1}. the continuous density $g$ is bounded away from zero.
\end{center}

\noindent Suppose we assume that \Cref{theo:main} holds for this restricted set of density functions, then for general continuous functions $g_0 \ge 0$ on $\overline{A}$ we have
\begin{align*}
\limsup_{t \to \infty} t\PP\left( \int_{A} g_0(x) \mu_{f}(dx) > t \right)
& \le \lim_{t \to \infty} t\PP\left( \int_{A} (g_0(x)+\epsilon) \mu_{f}(dx) > t \right)\\
& \le \frac{\int_{A} (g_0(v)+\epsilon) dv}{\sqrt{\pi d}}
\end{align*}

\noindent for arbitrary $\epsilon > 0$. On the other hand, the Lebesgue measure coincides with the Jordan inner content for any open sets, and we can find some elementary set (i.e. union of finitely many rectangles, which is Jordan measurable) $A_\epsilon \subset A$ such that $g_0|_{\overline{A}_{\epsilon}} \ge \epsilon$ and $\int_{A \setminus A_{\epsilon}} g_0(x) \le \epsilon$. This gives
\begin{align*}
\liminf_{t \to \infty} t\PP\left( \int_{A} g_0(x) \mu_{f}(dx) > t \right)
& \ge \lim_{t \to \infty} t\PP\left( \int_{A_\epsilon} g_0(x) \mu_{f}(dx) > t \right)\\
& \ge \frac{\int_{A_\epsilon} g_0(v)dv}{\sqrt{\pi d}}
\ge \frac{\int_{A} g_0(v) dv - \epsilon}{\sqrt{\pi d}}
\end{align*}

\noindent and hence  $\lim_{t \to \infty} t\PP\left( \int_{A} g_0(x) \mu_{f}(dx) > t \right) = (\pi d)^{-1/2} \int_A g_0(v) dv$.\\

Note that we may also assume without loss of generality that
\begin{center}
\textbf{R2}.  the function $f$ appearing in \eqref{eq:cov} is lower bounded by any constant.
\end{center}

\noindent Indeed, we can always rewrite the covariance kernel
\begin{align*}
\EE[X(x) X(y)] 
= - \log |x-y| + f(x, y)
= - \log |rx - ry| + f\left(\frac{rx}{r}, \frac{ry}{r}\right) + \log r
\end{align*}

\noindent with any $r > 1$, and introduce a rescaled log-correlated field $X^r(\cdot)$ on $rD = \{rx: x \in D\}$ with covariance structure given by
\begin{align*}
\EE[X^r(u) X^r(v)] 
= - \log |u - v| + f^r(u, v), \qquad f^r(u, v) = f(r^{-1}u, r^{-1} v) + \log r.
\end{align*}

\noindent where $f^r(u, v) \ge \log r - ||f||_\infty$.

\subsection{The splitting lemma}\label{subsec:split}
We explain a further reduction which allows us to assume both \textbf{R2} and
\begin{center}
\textbf{R3}. $A$ is contained in a Euclidean ball of arbitrarily small radius.
\end{center}

\noindent This obviously does not follow from a rescaling argument since the rescaling required for \textbf{R2} and for \textbf{R3} are in opposite directions. We need the splitting lemma below.

\begin{lem}\label{lem:split}
Let $A = A_+ \cup A_-$ be a partition of $A$ by some hyperplane, i.e. there exists some $a \in \RR^d$ and $c \in \RR$ such that 
\begin{align*}
A_+ = \{x \in A: \langle a, x \rangle \ge c\},
\qquad A_- = \{x \in A: \langle a, x \rangle < c\}.
\end{align*}

\noindent Then, as $\lambda \to 0^+$, we have
\begin{align*}
\EE\left[ 1 - e^{- \lambda \mu_{f, g}(A)^2}\right]
= \EE\left[ 1 - e^{- \lambda \mu_{f, g}(A_+)^2}\right]
+ \EE\left[ 1 - e^{- \lambda \mu_{f, g}(A_-)^2}\right]
+ o(\lambda^{1/2}).
\end{align*}

\end{lem}

\begin{proof}
We make use of the elementary inequality
\begin{align*}
1-e^{-(x+y)^2}
\begin{array}{l}
\ge \left(1-e^{-x^2} \right) + \left(1-e^{-x^2} \right) - \left(1-e^{-2xy} \right)\\
\le \left(1-e^{-x^2} \right)  + \left(1-e^{-x^2} \right) + \left(1-e^{-2xy} \right)
\end{array}
\qquad \forall x, y \ge 0.
\end{align*}

\noindent This means that
\begin{align*}
\EE\left[ 1 - e^{- \lambda \mu_{f, g}(A)^2}\right]
& = \EE\left[ 1 - e^{- \lambda \mu_{f, g}(A_+)^2}\right]
+ \EE\left[ 1 - e^{- \lambda \mu_{f, g}(A_-)^2}\right] \\
& \qquad + O \left(\EE\left[ 1 - e^{- 2\lambda \mu_{f, g}(A_+)\mu_{f, g}(A_-)}\right] \right) , \qquad \lambda \to 0^+.
\end{align*}

However, \Cref{lem:cross} suggests the existence of some $h > \frac{1}{2}$ such that 
\begin{align*}
\EE \left[\mu_{f, g}(A_+)^h \mu_{f, g}(A_-)^h\right] < \infty
\end{align*}

\noindent which implies $\PP(\mu_{f, g}(A_+)\mu_{f, g}(A_-) > t) = o(t^{-h})$ by Markov's inequality. A direct computation (or by \Cref{lem:reform}) then shows that $\EE\left[ 1 - e^{- 2\lambda \mu_{f, g}(A_+)\mu_{f, g}(A_-)}\right] = o(\lambda^{1/2})$ and we are done.
\end{proof}

The splitting lemma, while simple and based on \Cref{lem:cross}, was not known in \cite{BKNSW2015} and its power could not have been harnessed. For an illustration, we have

\begin{cor}\label{cor:partial}
Let $A =[0,a]^d \subset B(0, r_d(L))$, and $\mu_{L, g}(dx) = g(x) \mu_L(dx)$ be the critical GMC associated with the $L$-exact field and continuous density $g \ge 0$ on $A$. Then there exists some constant $\overline{C}_{d} > 0$ independent of $a$ such that
\begin{align}\label{eq:partial2a}
\PP\left(\mu_{L, g}(A) > t \right) \overset{t \to \infty}{\sim} \frac{\overline{C}_{d} \int_A g(v) dv}{t},
\end{align}

\noindent or equivalently
\begin{align}\label{eq:partial2b}
\lim_{\lambda \to 0^+} \EE \left[\frac{1- e^{-\lambda \mu_{L, g}(A)^2}}{\lambda^{1/2}}\right] =\overline{C}_d \sqrt{\pi} \int_A g(v) dv.
\end{align}
\end{cor}

\begin{proof}
We now further assume \textbf{R1}, i.e. $g(x)$ is a continuous function on $A$ and bounded away from zero. Note that $g$ is also uniformly continuous on $A$ by compactness. We can therefore consider a partition $A = \cup_i B_i$ into finitely many disjoint $d$-cubes $B_i$ of same length such that $g(\cdot)$ does not fluctuate multiplicatively by more than $\epsilon > 0$ on each of them, i.e.
\begin{align*}
\sup_{v \in B_i} g(v) \le (1+\epsilon) g(x) \le (1+\epsilon)^2 \inf_{v \in B_i} g(v) \qquad \forall x \in B_i.
\end{align*}

Writing $|B_i|$ for the Lebesgue measure of the cube $B_i$, we recall \Cref{lem:partial} and \Cref{cor:pscale} which say that there exists some $\overline{C}_d > 0$ such that
\begin{align*}
\PP\left(\mu_L(B_i) > t\right) \overset{t \to \infty}{\sim} \frac{\overline{C}_d |B_i|}{t}
\qquad \text{or} \qquad 
\lim_{\lambda \to 0^+} \EE \left[\frac{1- e^{-\lambda \mu_{L}(B_i)^2}}{\lambda^{1/2}}\right] = \overline{C}_d |B_i|\sqrt{\pi}
\end{align*}

\noindent (the equivalence comes from \Cref{lem:reform}) for all $i$. By \Cref{lem:split},
\begin{align*}
\limsup_{\lambda \to 0^+}  \EE \left[\frac{1- e^{-\lambda \mu_{L, g}(A)^2}}{\lambda^{1/2}}\right]
& \le \sum_i \limsup_{\lambda \to 0^+}  \EE \left[\frac{1- e^{-\lambda \mu_{L, g}(B_i)^2}}{\lambda^{1/2}}\right]\\
& \le \sum_i\limsup_{\lambda \to 0^+}  \EE \left[\frac{1- e^{-\lambda (\sup_{v \in B_i} g(v))^2 \mu_{L}(B_i)^2}}{\lambda^{1/2}}\right] \\
& =  \sum_i \overline{C}_{d} |B_i|   \left(\sup_{v \in B_i} g(v)\right) \sqrt{\pi}
\le  (1+\epsilon) \overline{C}_{d} \sqrt{\pi} \int_{A} g(v)dv.
\end{align*}

\noindent The inequality in the other direction
\begin{align*}
\liminf_{\lambda \to 0^+}  \EE \left[\frac{1- e^{-\lambda \mu_{L, g}(A)^2}}{\lambda^{1/2}}\right]
\ge (1+\epsilon)^{-1} \overline{C}_{d}  \sqrt{\pi} \int_{A} g(v)dv
\end{align*}

\noindent may be obtained similarly and this concludes the proof.
\end{proof}

\subsection{Universality of tail profile} \label{subsec:universal}
The goal of this subsection is to establish \Cref{theo:main}, modulo the identification of a proportionality constant.
\begin{lem}\label{lem:universal}
Under the setting of \Cref{theo:main}, we have
\begin{align}\label{eq:lap_universal}
\lim_{\lambda \to 0^+}  \EE \left[\frac{1- e^{-\lambda \mu_{f, g}(A)^2}}{\lambda^{1/2}}\right] =  \overline{C}_d \sqrt{\pi} \int_A g(v) dv,
\end{align}

\noindent or equivalently
\begin{align*}
\PP\left(\mu_{f, g}(A) > t \right) \overset{t \to \infty}{\sim} \frac{\overline{C}_{d} \int_A g(v) dv}{t}
\end{align*}

\noindent where $\overline{C}_d$ is the constant in \Cref{cor:partial}.
\end{lem}

Before we proceed to the proof, let us prepare ourselves by introducing some new notations and collecting some auxiliary lemmas below.

Let $B \subset D$ be an open $d$-cube with centre $v \in B$ and of length less than $r_d / d = r_d(0) / d$ (in particular it is contained in the Euclidean ball $B(v, r_d)$). We define for each $s \in [0,1]$ a new log-correlated Gaussian field 
\begin{align}\label{eq:inter_field}
Z_{B, s}(x) = \sqrt{s}X(x) + \sqrt{1-s} Y_{f(v, v)}(x), \qquad x \in B
\end{align} 

\noindent where $Y_{f(v, v)}$ is an $f(v, v)$-exact field\footnote{This is defined on $B(v, r_d)$ so long as $f(v, v) \ge 0$ for any $v$, which is justified by reduction \textbf{R2}.} independent of $X$. We abuse the notation and denote by $\mu_{B, s}(d\cdot)$ the critical GMC associated to $Z_{B, s}(\cdot)$.
\begin{lem}\label{lem:tailbound}
There exists some $C > 0$ independent of $B \subset D$ and $s \in [0,1]$, and $t_0 = t_0(B)$ possibly depending on the length of $B$, such that
\begin{align*}
& \PP(\mu_{B, s}(B) > t) \le \frac{C |B|}{t} \qquad  \forall t > 0\\
\text{and } \qquad & \PP(\mu_{B, s}(B) > t) \ge \frac{|B|}{Ct} \qquad \forall t > t_0.
\end{align*}
\end{lem}

\begin{proof}
We only prove the upper bound here since the lower bound is similar.

If $\mu_0(d\cdot)$ is the critical GMC associated to the exact field $Y_0$, we have by \Cref{cor:partial} that $\PP(\mu_0(B) > t) \sim \overline{C}_d |B| / t$, and so there exists some $C > 0$ such that 
\begin{align} \label{eq:upper_exact}
\PP\left(\mu_0(B) > t\right) \le \frac{C}{t} \qquad \forall t > 0
\end{align}

\noindent for some cube $B$ of fixed size contained in a ball of radius $r_d = r_d(0)$.

Let $G_-(x)$ be an independent Gaussian field with covariance kernel $f_-(x, y)$ on $\overline{D}$. By inspecting the covariance structure of $Z_{B, s}(\cdot)$
\begin{align*}
\EE\left[Z_{B, s}(x)Z_{B, s}(y)\right] = -\log |x-y| + sf_+(x, y) - sf_-(x, y) + (1-s) f(v, v),
\end{align*}

\noindent we see that 
\begin{align*}
Z_{B, s}(\cdot) + \sqrt{s} G_-(\cdot) \overset{d}{=} Y_0(\cdot) + \sqrt{s}G_+(\cdot) + N_{s, v} \qquad \text{on } B
\end{align*} 

\noindent where $Y_0(\cdot)$ is an exact log-correlated Gaussian field, $G_+(\cdot)$ is an independent Gaussian field with covariance $f_+(x, y)$, and $N_{s, v}$ is an independent Gaussian random variable with variance $(1-s)f(v, v)$ (where $f(v,v)$ is assumed to be non-negative by reduction \textbf{R2}). 

Now fix some $a > 0$ such that $\PP(\min_{x \in D} e^{\sqrt{2ds} G_-(x) - ds \EE[G_-(x)^2]} > a) \ge C_a$ for some $C_a > 0$ bounded away from zero uniformly in $s \in [0, 1]$, which is possible by \Cref{lem:ctsGP}. Then for any $c \in (0, 1)$ we have
\begin{align*}
\PP(\mu_{B, s}(cB) > t) 
&\le \frac{1}{C_a} \PP\left(\min_{x \in D} e^{\sqrt{2ds} G_-(x) - ds \EE[G_-(x)^2]}\mu_{B, s}(cB) > at\right)\\
&\le \frac{1}{C_a} \PP\left(\max_{x \in D} e^{\sqrt{2ds}  G_+(x) - ds \EE[G_+(x)^2]}e^{\sqrt{2d} N_{s, v}- d\EE[N_{s, v}^2]} \mu_{0}(cB) > at\right),
\end{align*}

\noindent and we may further rewrite $\mu_0(cB) \overset{d}{=} c^d e^{\sqrt{2d}N_c - d \EE[N_c^2]} \mu_0(B)$ where $N_c$ is an independent $N(0, - \log c)$ variable, using the scaling property of the exact field $Y_0$. From \Cref{lem:ctsGP} we know that e.g. the second moment of the random variable
\begin{align*}
\max_{x \in D} e^{\sqrt{2ds} G_+(x) - ds \EE[G_+(x)^2]}e^{\sqrt{2d} N_{s, v}- d\EE[N_{s, v}^2]}c^d e^{\sqrt{2d}N_c - d \EE[N_c^2]}
\end{align*}

\noindent may be bounded uniformly in $s \in [0, 1]$ and $v \in D$, and by \eqref{eq:upper_exact} and \Cref{lem:aux} we conclude that
\begin{align*}
\PP(\mu_{B, s}(cB) > t) 
& \le \frac{C \EE \left[\max_{x \in D} e^{\sqrt{2ds} G_+(x) - ds \EE[G_+(x)^2]}e^{\sqrt{2d} N_{s, v}- d\EE[N_{s, v}^2]}c^d e^{\sqrt{2d}N_c - d \EE[N_c^2]}\right]}{C_a a t}\\
& \le \frac{C_a' |cB|}{t} \qquad \forall t > 0.
\end{align*}
\end{proof}

\begin{proof}[Proof of \Cref{lem:universal}]
Let us start by assuming that $g \equiv 1$. Recall that the Lebesgue measure of an open set $A$ coincides with its inner Jordan measure. In other words, for each $\delta > 0$ there exists some elementary set $A_\delta = \cup_i A_{\delta, i}$ formed by finite union of disjoint cubes $A_{\delta, i}$ such that
\begin{align*}
A_\delta \subset A \qquad \text{and} \qquad |A \setminus A_\delta| \le \delta
\end{align*}

\noindent If we write $B_\delta = A \setminus \overline A_\delta$, then \Cref{lem:split} says
\begin{align*}
\EE \left[\frac{1- e^{-\lambda \mu_{f}(A)^2}}{\lambda^{1/2}}\right]
& = \EE \left[\frac{1- e^{-\lambda \mu_{f}(A_\delta)^2}}{\lambda^{1/2}}\right]
+ \EE \left[\frac{1- e^{-\lambda \mu_{f}(B_\delta)^2}}{\lambda^{1/2}}\right]\\
& \qquad + O\left( \EE \left[\frac{1- e^{-2\lambda \mu_{f}(A_\delta)\mu_{f}(B_\delta)}}{\lambda^{1/2}}\right]\right)
\end{align*}

\noindent where the last term is $o(1)$ as $\lambda \to 0^+$ and may be safely ignored\footnote{\Cref{lem:split} requires our sets to be separated by hyperplanes which is not immediately satisfied here, but we may always subdivide $B_\delta$ into finitely many smaller sets so that each subdivision is separated from $A_\delta$ by a hyperplane.}.

We first treat $A_\delta = \cup_i A_{\delta, i}$, and may assume that the cubes $A_{\delta, i}$ are small enough (by further subdivision) such that the function $f$ appearing in the covariance \eqref{eq:cov} does not fluctuate by more than $\delta$ on each of them, i.e.
\begin{align*}
|f(x_1, y_1) - f(x_2, y_2)| \le \delta
\qquad \forall x_1, y_1, x_2, y_2 \in A_{\delta, i}, \qquad \forall i.
\end{align*}

Let $v_i \in A_{\delta, i}$ be the centre of the cube. We introduce the interpolation field $Z_{A_{\delta, i}, s}$ as in \eqref{eq:inter_field} and the critical GMC $\mu_{A_{\delta, i}, s}(d\cdot)$ accordingly. By \Cref{cor:interpolate}
\begin{align*}
\exp\left(-\delta \int_0^1 \frac{\EE \left[ (W_s)^2 |F''(W_s)|\right]}{\EE \left[F(W_s)\right]} ds\right)
\le \frac{\EE \left[1- e^{-\lambda W_1^2} \right]}{\EE \left[1- e^{-\lambda W_0^2} \right]}
\le \exp\left(\delta \int_0^1 \frac{\EE \left[ (W_s)^2 |F''(W_s)|\right]}{\EE \left[F(W_s)\right]} ds\right)
\end{align*}

\noindent where 
\begin{align*}
F(x) = 1 - e^{-\lambda x^2},
\qquad x^2|F''(x)| \le e^{-\lambda x^2}\left(1 + 8 \lambda^2 x^4 \right)
\end{align*}

\noindent are both bounded continuous functions in $x$, and $W_s = \mu_{A_{\delta, i}, s}(A_{\delta, i})$. Using the two-sided bounds from \Cref{lem:tailbound}, it is a straightforward computation to show that there exists some $C>0$ independent of $s \in [0,1]$ and $A_{\delta, i}$ such that for $\lambda$ sufficiently large\footnote{say $\lambda > \lambda_0$, where $\lambda_0$ depends on $t_0(A_{\delta, i})$ from \Cref{lem:tailbound}. For each $\delta > 0$ since the number of cubes $A_{\delta, i}$ we deal with is finite, $\lambda_0 = \lambda_0(\delta)$ may be taken uniformly for all $i$. While the bound \eqref{eq:interbound} depends on $\delta$ through $\lambda > \lambda_0(\delta)$, the constant $C>0$ on the RHS is independent of $\delta$.}
\begin{align}\label{eq:interbound}
\frac{\EE \left[ (W_s)^2 |F''(W_s)|\right]}{\EE \left[F(W_s)\right]} \le C.
\end{align}

\noindent But then 
\begin{align*}
\mu_{A_{\delta, i}, 0}(A_{\delta, i}) = \mu_{f \equiv f(v_i, v_i)}(A_{\delta, i})
\qquad \text{and} \qquad \mu_{A_{\delta, i}, 1}(A_{\delta, i}) = \mu_{f}(A_{\delta, i}),
\end{align*}

\noindent and we have by \Cref{lem:split} that
\begin{align*}
\lambda^{-1/2} \EE \left[1 - e^{-\lambda \mu_f(A_\delta)^2}\right]
& =   \sum_{i} \lambda^{-1/2}\EE \left[1 - e^{-\lambda \mu_f(A_{\delta, i})^2}\right] + o(1)\\
& \ge (1-C \delta) \sum_{i} \lambda^{-1/2}\EE \left[ 1 - e^{-\lambda \mu_{f \equiv f(v_i, v_i)} (A_{\delta, i})^2}\right] + o(1)\\
& = (1-C \delta) \sum_{i} \overline{C}_d \sqrt{\pi} |A_{\delta, i}|+ o(1)
= (1-C \delta)  \overline{C}_d \sqrt{\pi} |A_{\delta}|+ o(1),
\end{align*}

\noindent and similarly
\begin{align*}
\lambda^{-1/2}\EE \left[1 - e^{-\lambda \mu_f(A_\delta)^2}\right]
\le (1+C \delta)  \overline{C}_d \sqrt{\pi} |A_{\delta}|+ o(1).
\end{align*}

Since $A$ and $\overline{A}_{\delta}$ are Jordan measurable, so is the set $B_{\delta} = A \setminus \overline{A}_{\delta}$. This means one can find an elementary set $\widetilde{B}_{\delta} \subset B_{\delta}$ such that $|\widetilde{B}_{\delta} \setminus B_{\delta}| \le \delta$, and similar calculation shows that
\begin{align*}
\lambda^{-1/2}\EE \left[1 - e^{-\lambda \mu_f(B_\delta)^2}\right]
\le C |\widetilde{B}_{\delta}|+ o(1)
= 2C  \delta + o(1).
\end{align*}

\noindent Putting all the pieces together, we obtain
\begin{align*}
(1-C \delta)  \overline{C}_d \sqrt{\pi} |A_{\delta}|
& \le \liminf_{\lambda \to 0^+} \lambda^{-1/2} \EE \left[1 - e^{-\lambda \mu_f(A)^2}\right]\\
& \le \limsup_{\lambda \to 0^+} \lambda^{-1/2} \EE \left[1 - e^{-\lambda \mu_f(A)^2}\right]
\le (1+C \delta)  \overline{C}_d \sqrt{\pi} |A_{\delta}| +  2C  \delta.
\end{align*}

\noindent Since $\delta > 0$ is arbitrary, we obtain \eqref{eq:lap_universal} when $g \equiv 1$. In general, we may assume that $g$ is bounded away from zero by reduction \textbf{R1}, and apply the same splitting argument in the proof of \Cref{cor:partial} to extend the tail asymptotics to any continuous densities $g(x) \ge 0$ on $\overline{A}$.

\end{proof}

\subsection{The constant $\overline{C}_d$} \label{subsec:eval}
We recall the following convenient fact which was also used in \cite{Won2019}: if $U$ is a non-negative random variable satisfying $\PP(U > t) \sim C/t$ as $t \to \infty$, then
\begin{align*}
\lim_{\lambda \to 0^+} \frac{\EE\left[U e^{-\lambda U}\right]}{-\log \lambda} = C.
\end{align*}

\noindent We shall take $U = \mu_L(A)$ where $\mu_L(d\cdot)$ is the critical GMC associated with the $L$-exact field $Y_L(\cdot)$, and $A = B(0, 2r)$ for some $r \in (0, r_d(L) \wedge \frac{1}{2})$, and we hope to identify $\overline{C}_d$ via the identity

\begin{align}\label{eq:eval0}
\lim_{\lambda \to 0^+} \frac{\EE\left[\mu_L(A) e^{-\lambda \mu_L(A)}\right]}{-\log \lambda} = \overline{C}_d |A|.
\end{align}

By the localisation trick, we have
\begin{align}\label{eq:eval1}
\EE \left[ \mu_{L}(A) e^{-\lambda \mu_L(A)} \right]
& = \lim_{\epsilon \to 0^+}  \int_A \left(\log \frac{1}{\epsilon}\right)^{1/2}\EE \left[ e^{-\lambda \mu_{L, \epsilon}(v, A)} \right] dv
\end{align}

\noindent where
\begin{align}\label{eq:eval2}
\mu_{L, \epsilon}(v, A) &= \left(\log \frac{1}{\epsilon}\right)^{1/2} \int_A \frac{e^{\sqrt{2d} Y_{L, \epsilon}(x) - d \EE\left[Y_{L, \epsilon}(x)^2\right]}}{\left(|x-v|\vee \epsilon\right)^{2d}} e^{2d L}dx.
\end{align}

\noindent Given that the constant $\overline{C}_d$ does not depend on $L$, \eqref{eq:eval0} suggests that the multiplicative factor $e^{2dL}$ in the integrand of \eqref{eq:eval2} in the definition of $\mu_{L, \epsilon}(v, A)$ may be safely omitted without affecting the limit (as we may always absorb this factor into $\lambda$).

\subsubsection{Gaussian comparison and the evaluation of localisation limit}
Let us introduce our reference log-correlated Gaussian field $X_d$ for $d \ge 2$ (see \Cref{rem:1d} for $d=1$) which is characterised by the covariance kernel
\begin{align*}
\EE[X_d(x) X_d(y)] = - \log |x-y| - S_d(x, y), \qquad \forall x, y \in B(0, 1)
\end{align*}

\noindent where
\begin{align}\label{eq:def_Sd}
S_d(x, y) = -\frac{|\mathbb{S}^{d-2}|}{2|\mathbb{S}^{d-1}|} \int_{-1}^1 (1-u^2)^{\frac{d-2}{2}}\log |1 -2cu + c^2|  du,
\qquad c = \frac{|x|}{|y|} \wedge \frac{|y|}{|x|}.
\end{align}

\noindent The existence of this log-correlated field and the associated critical GMC $\mu_d$ are discussed in \Cref{app:ref}. Here we would like to note that this is just the exact field when $d=2$ because $S_2(x, y) \equiv 0$, whereas in general $S_d(x, y)$
\begin{itemize}
\item is continuous and bounded uniformly in $x$ and $y$ everywhere except at $(x, y) = (0, 0)$ where the function is not defined;
\item is rotation invariant in each variable, and jointly scale invariant in the sense that $S_d(ax, ay) = S_d(x, y)$ for any $a > 0$.
\end{itemize}

By the uniform boundedness of $S_d(\cdot, \cdot)$, we can choose $L_0 < L_1$ and $A = B(0, 2r) \subset B(0, r_d(L_0)) \subset B(0, r_d(L_1)$ such that
\begin{align*}
\EE[Y_{L_0}(x) Y_{L_0}(y)] \le \EE[X_d(x) X_d(y)] \le \EE[Y_{L_1}(x) Y_{L_1}(y)] \qquad \forall x, y \in A
\end{align*}

\noindent and this comparison immediately extends to the mollified fields $Y_{L_0, \epsilon}, X_{d, \epsilon}$ and $Y_{L_1, \epsilon}$. Applying Kahane's convexity inequality \eqref{eq:Gcomp} to the RHS of \eqref{eq:eval1}, we may study instead the limit
\begin{align}
\label{eq:lim_fusion}
& \lim_{\lambda \to 0^+} \frac{1}{-\log \lambda} \left[\lim_{\epsilon \to 0^+}  \int_A \left(\log \frac{1}{\epsilon}\right)^{1/2}\EE \left[ e^{-\lambda \mu_{d , \epsilon}(v, A)} \right] dv\right]\\
\label{eq:lim_fusion2} 
\text{where }\qquad & \mu_{d, \epsilon}(v, A) = \left(\log \frac{1}{\epsilon}\right)^{1/2} \int_A \frac{e^{\sqrt{2d} X_{d, \epsilon}(x-v) - d \EE\left[X_{d, \epsilon}(x-v)^2\right]}}{\left(|x-v|\vee \epsilon\right)^{2d}} dx.
\end{align}

The sequence of random variables $\mu_{d, \epsilon}(v, A)$ actually blows up in the limit as $\epsilon \to 0^+$, which is expected given the extra factor of $(\log 1 / \epsilon)^{1/2}$ in \eqref{eq:lim_fusion}. The blow-up is, not surprisingly, due to the singularity near $x = v$, and it is therefore useful to split the integral into two parts depending on whether $x \in A \cap B(v, r)$ or $x \in A \cap B(v, r)^c$ for further analysis. The latter contribution, i.e.
\begin{align*}
R_{d, \epsilon}(v, r)
& = \left(\log \frac{1}{\epsilon}\right)^{1/2} \int_{A \cap B(v, r)^c} \frac{e^{\sqrt{2d} X_{d, \epsilon}(x-v) - d \EE\left[X_{d, \epsilon}(x-v)^2\right]}}{\left(|x-v|\vee \epsilon\right)^{2d}} dx,
\end{align*}

\noindent does not involve any singularity and actually converges to some finite mass 
\begin{align*}
R_{d}(v, r) = \int_{(A - v) \cap B(0, r)^c} |x|^{-2d} \mu_d(dx)
\end{align*}

\noindent as $\epsilon \to 0^+$ by the construction of critical GMCs.

Fix $\lambda > 0$. It shall be self-evident from our analysis that the integrand in \eqref{eq:lim_fusion} is bounded uniformly in $v \in A$ as $\epsilon$ tends to zero, and by dominated convergence we may interchange the order of the $\epsilon$-limit and integration provided that:
\begin{lem}\label{lem:fusion}
We have
\begin{align}\label{eq:fusion}
\lim_{\epsilon \to 0^+}\left(\log \frac{1}{\epsilon}\right)^{1/2} \EE \left[ e^{-\lambda \mu_{d , \epsilon}(v, A)} \right]
=\sqrt{\frac{2}{\pi}} \int_0^\infty \EE \left[e^{-\lambda\left(e^{\sqrt{2d}x} \widetilde{\mu}_{d}^x(v, r) + R_{d}(v, r)  \right) }\right] dx
\end{align}

\noindent where
\begin{align} \label{eq:fusion_main}
\widetilde{\mu}_d^x (v, r) = e^{\sqrt{2d}B_{-\log r}} \int_{-L_{x, -}}^\infty e^{-\sqrt{2d}\beta_s} Z_d^{A, v} \circ \phi_{- \log r + L_{x, -}}(ds)
\end{align}

\noindent with
\begin{itemize}
\item $Z_d^{A, v}(ds) = \int_{\mathbb{S}^{d-1}} 1_A(v+e^{-s}x) e^{\sqrt{2d}\widehat{Y}_d(e^{-s}x) - d\EE\left[\widehat{Y}_d(e^{-s}x)^2\right]} \sigma(dx) ds$ where $\sigma(dx)$ is the uniform measure on the $(d-1)$-sphere and $\widehat{Y}_d$ is defined in \eqref{eq:Yhat}.
\item $\phi_c: s \mapsto s+c$ is the shift operator.
\item $B_{- \log r} = \overline{X}(r) \sim N(0, -\log r)$.
\item $(\beta_s)_{s \ge 0}, (\beta_{-s})_{s \ge 0}$ are two independent $\mathrm{BES}_0(3)$-process.
\item $L_{x, -} := \sup\{s \ge 0: \beta_{-s} = x\}$.
\end{itemize}
\end{lem}

The evaluation of the limit \eqref{eq:fusion} borrows techniques from the study of fusion estimates in Gaussian multiplicative chaos \cite{BW2018}, and we postpone the sketch of proof of \Cref{lem:fusion} to \Cref{app:fusion}. 

\subsubsection{Identifiying the proportionality constant}
Our remaining task is the evaluation of the limit
\begin{align*}
\lim_{\lambda \to 0^+} \frac{1}{-\log \lambda} \int_0^\infty \EE \left[e^{-\lambda\left(e^{\sqrt{2d}x} \widetilde{\mu}_{d}^x(v, r) + R_{d}(v, r)  \right) }\right] dx.
\end{align*}

\noindent We first show that $\widetilde{\mu}_d^x(v, r)$ and $R_d(v, r)$ possess moments of all order smaller than $1$:
\begin{lem}\label{lem:fusion_moment}
For each $q < 1$, we have
\begin{align*}
\EE \left[R_d(v, r)^q \right] < \infty
\qquad \EE \left[ \widetilde{\mu}_d^x(v, r)^q \right] < \infty
\end{align*}

\noindent where the bounds may be taken uniformly in $v \in \overline{A}$ and $x \ge 0$. In particular,
\begin{align*}
& \left|\EE \left[\log \left(\widetilde{\mu}_d^x(v, r) + R_d(v, r)\right) \right]\right| \\
&\qquad  \le 2\EE \left[\left(\widetilde{\mu}_d^x(v, r) + R_d(v, r)\right)^{1/2}+ \left(\widetilde{\mu}_d^x(v, r) + R_d(v, r)\right)^{-1/2} \right] 
< \infty
\end{align*}

\noindent uniformly in $v \in \overline{A}$ and $x \ge 0$.
\end{lem}

\begin{proof}
We only need to treat $\widetilde{\mu}_d^x(v, r)$ since the statement regarding $R_d(v, r)$ follows immediately by the existence of GMC moments from \Cref{lem:GMC_moment} (and a uniform bound may be obtained if we take $v \in \partial A = \{x: |x| = 2r\}$ for negative moments, or $v = 0$ for positive moments).

Let us commence with the negative moments $q < 0$. From the construction in \Cref{app:fusion} we have
\begin{align*}
\EE\left[\widehat{Y}_d(x)\widehat{Y}_d(y)\right] = \log \frac{|x| \vee |y|}{|x-y|} - S_d(x, y),
\end{align*}

\noindent showing that $\widehat{Y}(u)$ is scale invariant in $u \in \RR^d$, or equivalently $\widehat{Y}_d(e^{-s}x)$ is translation invariant in $s > 0$. Combining this with the observation that $1_A(v+ e^{-s} x) \le 1_A(v+ e^{-(s+c)}x)$ for any $c > 0$ (since $A = B(0, 2r)$ here), we see that $Z_d^{A, v}\circ \phi_{- \log r + L_{x, -}} (ds)$ stochastically dominates $Z_d^{A,v}\circ \phi_{-\log r} (ds)$. It is therefore sufficient to show that
\begin{align*}
\EE \left[ \left(\int_{0}^\tau e^{-\sqrt{2d}\beta_s} Z_d^{A, v}\circ \phi_{-\log r}  (ds)\right)^q \right] < \infty
\end{align*}

\noindent where $\tau := \inf \{s \ge 0: \beta_s \ge 4\}$. Viewing our Bessel process $(\beta_s)_{s \ge 0}$ as the evolution of the Euclidean norm of a $3$-dimensional Brownian motion $(B_s^1, B_s^2, B_s^3)_{s \ge 0}$, we have
\begin{align*}
\EE \left[ \left(\int_{0}^\tau e^{-\sqrt{2d}\beta_s} Z_d^{A, v}\circ \phi_{-\log r} (ds)\right)^q \right] 
& \le \sum_{i=1}^3\EE\left[ \left(\int_{0}^{\tau_i} e^{-\sqrt{2d}\sqrt{|B_s^i|^2 + 2}} Z_d^{A, v}\circ \phi_{-\log r} (ds)\right)^{q}\right]\\
& \le 3e^{2q\sqrt{d}} \EE\left[ e^{q\sqrt{2d} \sup_{s \le \tau_1} |B_s^1|}\left(\int_{0}^{\tau_1}Z_d^{A, v}\circ \phi_{-\log r} (ds)\right)^{q}\right]
\end{align*}

\noindent where $\tau_i := \min(1, \inf \{s \ge 0: |B_s^i| \ge 1 \})$. But then
\begin{align}
\notag & \EE\left[ e^{q\sqrt{2d} \sup_{s \le \tau_1} |B_s^1|}\left(\int_{0}^{\tau_1}Z_d^{A, v}\circ \phi_{-\log r} (ds)\right)^{q}\right]\\
\notag & \quad \le  \EE\left[ e^{q\sqrt{2d} \sup_{s \le 1} |B_s^1|}\right]  \EE\left[\left(\int_{0}^{1}Z_d^{A, v}\circ \phi_{-\log r} (ds)\right)^{q}\right]\\
\notag & \quad \qquad + \sum_{k \ge 0}  \EE\left[1_{\{2^{-k-1} \le \tau_1 \le 2^{-k}\}}e^{q\sqrt{2d} \sup_{s \le 2^{-k}} |B_s^1|}\left(\int_{0}^{2^{-k}}Z_d^{A, v}\circ \phi_{-\log r} (ds)\right)^{q}\right]\\
\label{eq:nmom_bound}
& \quad \le C + \sum_{k \ge 0} \PP\left(2^{-k-1} \le \tau_1 \le 2^{-k}\right)^{1/2} \EE \left[e^{2q\sqrt{2d} \sup_{s \le 1}|B_s^1|}\right] \EE\left[\left(\int_{0}^{2^{-k}}Z_d^{A, v}\circ \phi_{-\log r} (ds)\right)^{q}\right].
\end{align}

\noindent We now rewrite $\int_0^{2^{-k}} Z_d^{A, v}(ds)$ in terms of Euclidean coordinates:
\begin{align*}
\int_0^{2^{-k}} Z_d^{A, v}\circ \phi_{-\log r} (ds) 
& = \int_{(A-v) \cap \{re^{-2^{-k}} \le |u| \le r\}} |u|^{2d} e^{\sqrt{2d} \widehat{Y}_d(u) - d \EE\left[\widehat{Y}_d(u)^2\right]} du.
\end{align*}

By the rotational invariance of $\widehat{Y}_d(u)$, we may assume without loss of generality that $v$ lies on the negative $e_1$-axis. It is then not hard to show that the ball centred at $\frac{r}{2} (1+e^{-2^{-k}}) e_1$ of radius $\frac{r}{1000}2^{-k}$ lies inside the intersection $B(-v, 2r) \cap \{re^{-2^{-k}} \le |u| \le r\}$, i.e.
\begin{align*}
\int_0^{2^{-k}} Z_d^{A, v}\circ \phi_{-\log r} (ds) 
& \ge (r/4)^{2d} \int_{B(\frac{r}{2} (1+e^{-2^{-k}}) e_1,  \frac{r}{1000}2^{-k})} e^{\sqrt{2d} \widehat{Y}_d(u) - d \EE\left[\widehat{Y}_d(u)^2\right]} du
\end{align*}

\noindent since $|u|\ge re^{-2^{-k}} \ge r/4$ for $k \ge 0$. We apply the scaling property of GMC moments from  \Cref{lem:GMC_moment}) (ii) and obtain
\begin{align*}
\EE\left[\left(\int_{0}^{2^{-k}}Z_d^{A, v}\circ \phi_{-\log r}(ds)\right)^{q}\right] \le C 2^{-k (2dq - dq^2)}
\end{align*} 

\noindent for some constant $C > 0$ independent of $v$. We also have
\begin{align*}
\PP\left(2^{-k-1} \le \tau_1 \le 2^{-k}\right)
\le \PP\left(\max_{s \le 2^{-k}} B_s^1 > 1\right)
\le C e^{-c 2^{k}},
\end{align*}

\noindent which implies that \eqref{eq:nmom_bound} is summable and hence all negative moments exist.

For $q \in [0, 1)$, we only need to show that 
\begin{align*}
\EE \left[ \left(\int_{1}^\infty e^{-\sqrt{2d}\beta_s} Z_d(ds)\right)^q \right] < \infty
\end{align*}

\noindent where $Z_d(ds)$ is defined in the same way as $Z_d^{A, v}(ds)$ except without the indicator function $1_A(v + e^{-s}x)$. Denote by $\EE_{Z_d}$ and $\EE_{\beta}$ the expectation over $Z_d(d\cdot)$ and $(\beta_s)_s$ respectively. By Jensen's inequality,
\begin{align*}
\EE \left[ \left(\int_{1}^\infty e^{-\sqrt{2d}\beta_s} Z_d(ds)\right)^q \right]
& \le \EE_{Z_d}\left[   \left(\EE_{\beta}\left[\int_{1}^\infty e^{-\sqrt{2d}\beta_s} Z_d(ds)\right]\right)^q \right]\\
& = \EE_{Z_d}\left[   \left(\int_{1}^\infty \EE_{\beta}\left[e^{-\sqrt{2d}\beta_s}\right] Z_d(ds)\right)^q \right].
\end{align*}

\noindent For each $s \ge 0$, we have $\beta_s \overset{d}{=} \sqrt{s}\chi_3$ where $\chi_3$ is a chi(3)-random variable, and so
\begin{align*}
\EE_{\beta}\left[e^{-\sqrt{2d}\beta_s}\right]
= \frac{1}{\sqrt{2}\Gamma(3/2)}\int_0^\infty e^{-\sqrt{2ds}x} x^2 e^{-x^2 / 2} dx
\le C s^{-3/2}.
\end{align*}

\noindent Thus,
\begin{align*}
\EE_{Z_d}\left[   \left(\int_{1}^\infty \EE_{\beta}\left[e^{-\sqrt{2d}\beta_s}\right] Z_d(ds)\right)^q \right]
& \le  C\EE\left[   \left(\int_{1}^\infty s^{-3/2} Z_d(ds)\right)^q \right]\\
& \le C \sum_{n=1}^\infty n^{-\frac{3q}{2}} \EE\left[   \left(\int_{n}^{n+1} Z_d(ds)\right)^q \right]
\end{align*}

\noindent which is summable if we choose any $q \in [0, 1)$ such that $3q/2 > 1$, as $\EE\left[   \left(\int_{n}^{n+1} Z_d(ds)\right)^q \right]$ does not depend on $n$ by translation invariance of $Z_d(ds)$ (which is inherited from the translation invariance of $\widehat{Y}_d(e^{-s}x)$ in $s$) and is finite by \Cref{lem:GMC_moment} for the existence of moments of critical GMCs.
\end{proof}

We are now ready to evaluate the constant $\overline{C}_d$.
\begin{lem}\label{lem:limit_result}
We have
\begin{align}\label{eq:limit_result}
\lim_{\lambda \to 0^+} \frac{1}{-\log \lambda}\int_0^\infty \EE \left[e^{-\lambda\left(e^{\sqrt{2d}x} \widetilde{\mu}_{d}^x(v, r) + R_{d}(v, r)  \right) }\right]dx = \frac{1}{\sqrt{2d}}.
\end{align}

\noindent In particular, the constant $\overline{C}_d$ is equal to $(\pi d)^{-1/2}$.
\end{lem}

\begin{proof}
Since $\widetilde{\mu}_d^x(v, r)$ is strictly increasing in $x$ (in the sense of stochastic order), we have
\begin{align*}
\int _{x_0}^\infty &  \EE \left[\exp\left\{-\lambda e^{\sqrt{2d}x}\left( \widetilde{\mu}_{d}^{\infty} (v, r) + \epsilon R_d(v, r) \right)\right\} \right] dx\\
&\qquad \le \int _0^\infty  \EE \left[\exp\left\{-\lambda\left(e^{\sqrt{2d}x} \widetilde{\mu}_{d}^x(v, r) + R_{d}(v, r)  \right)\right\}\right] dx\\
& \qquad \qquad  \qquad \le x_0 + \int _{x_0}^\infty  \EE \left[\exp\left\{-\lambda\left(e^{\sqrt{2d}x} \widetilde{\mu}_{d}^{x_0} (v, r) \right) \right\}\right]dx
\end{align*}

\noindent where $x_0 > 0$ and $\epsilon = e^{-\sqrt{2d}x_0}$.

For generic non-negative random variable $U$, we have
\begin{align}
\notag &\int_{x_0}^\infty \EE \left[ \exp \left\{-\lambda e^{\sqrt{2d}x}U\right\}\right] dx
= \frac{1}{\sqrt{2d}}\EE\left[\int_{\lambda e^{\sqrt{2d}x_0}U}^\infty e^{-y} \frac{dy}{y}\right]\\
\label{eq:terms}
& = \frac{1}{\sqrt{2d}}\left\{
-\EE\left[ e^{-\lambda e^{\sqrt{2d}x_0}U} \log \left(\lambda e^{\sqrt{2d}x_0}U\right)\right]
+ \EE\left[\int_{\lambda e^{\sqrt{2d}x_0}U}^\infty e^{-y}\log y  dy\right]
\right\}.
\end{align}

\noindent Therefore, 
\begin{align*}
& \int _{x_0}^\infty \EE \left[\exp\left\{-\lambda e^{\sqrt{2d}x}\left( \widetilde{\mu}_{d}^{\infty} (v, r) + \epsilon R_d(v, r) \right)\right\} \right] dx\\
& \qquad \ge  \frac{-\log \lambda }{\sqrt{2d}}\EE\left[ e^{-\lambda e^{\sqrt{2d}x_0}(\widetilde{\mu}_{d}^{\infty} (v, r) + \epsilon R_d(v, r))}\right] \\
& \qquad \quad -\underbrace{\frac{1}{\sqrt{2d}} \left\{\EE\left[ \sqrt{2d}x_0 + \log \left(\widetilde{\mu}_d^\infty(v, r) + \epsilon R_d( v, r)\right)\right] + \int_0^\infty e^{-y}|\log y|  dy \right\}}_{O(1)}
\end{align*}

\noindent and
\begin{align*}
&  \int _{x_0}^\infty  \EE \left[\exp\left\{-\lambda\left(e^{\sqrt{2d}x} \widetilde{\mu}_{d}^{x_0} (v, r) \right) \right\}\right]dx\\
& \qquad \qquad \le \frac{-\log \lambda }{\sqrt{2d}} + \underbrace{ \frac{1}{\sqrt{2d}} \left\{\EE \left[ \log \widetilde{\mu}_d^{x_0}(v, r)\right] + \int_0^\infty e^{-y} \log|y|dy  \right\}}_{O(1)}
\end{align*}

\noindent and hence
\begin{align*}
\lim_{\lambda \to 0^+} \frac{1}{-\log \lambda}\int_0^\infty \EE \left[e^{-\lambda\left(e^{\sqrt{2d}x} \widetilde{\mu}_{d}^x(v, r) + R_{d}(v, r)  \right) }\right]dx = \frac{1}{\sqrt{2d}}.
\end{align*}

The above analysis also easily gives
\begin{align*}
\frac{1}{-\log \lambda}\int_0^\infty \EE \left[e^{-\lambda\left(e^{\sqrt{2d}x} \widetilde{\mu}_{d}^x(v, r) + R_{d}(v, r)  \right) }\right]dx \le C
\end{align*}

\noindent for some constant $C > 0$ independent of $v \in A$ and $\lambda \in (0, 1/2)$, thanks to the uniform bounds from \Cref{lem:fusion_moment}. To finish our proof, recall that
\begin{align*}
\overline{C}_d|A| 
& = \lim_{\lambda \to 0^+} \frac{\EE\left[\mu_L(A) e^{-\lambda \mu_L(A)}\right]}{-\log \lambda}\\
& = \sqrt{\frac{2}{\pi}} \lim_{\lambda \to 0^+}\frac{1}{-\log \lambda}  \int_A  \left(\int_0^\infty \EE \left[e^{-\lambda\left(e^{\sqrt{2d}x} \widetilde{\mu}_{d}^x(v, r) + R_{d}(v, r)  \right) }\right] dx\right) dv\\
& = \sqrt{\frac{2}{\pi}} \int_A  \frac{1}{\sqrt{2d}} dv = \frac{1}{\sqrt{\pi d}} |A|
\end{align*}

\noindent by dominated convergence, from which we conclude that $\overline{C}_d = (\pi d)^{-1/2}$. 
\end{proof}

\begin{rem}\label{rem:1d}
The analysis can be performed ad verbatim when $d=1$, except that we pick the exact field $Y_0(\cdot)$ on the interval $[-1,1]$ as our reference field $X_d(\cdot)$, and that $Z_{d}^{A, v}(ds)$ is defined according to the substitution of variable $e^{-s} = |x- v|$.
\end{rem}

\appendix
\section{Proof of \Cref{lem:cross}}\label{app:cross}
The goal of this appendix is to give a proof of
\begin{align*}
\EE\left[\mu_{f, g}(B_1)^h \mu_{f, g}(B_2)^h \right] < \infty.
\end{align*}

We may assume without loss of generality that $B_1, B_2 \in B(0, r_d(0))$ using a basic scaling argument, and prove the above claim for $g \equiv 1$ and $f \equiv 0$ as the general case follows from the following basic consideration: if $G_-(\cdot)$ is an independent continuous Gaussian field on $\overline{D}$ with covariance $f_-(\cdot, \cdot)$ on $\overline{D} \times \overline{D}$, then
\begin{align*}
(X(x) + G_-(x))_{x \in D} \overset{d}{=} (Y_0(x) + G_+(x))_{x \in D}
\end{align*}

\noindent where $Y_0(\cdot)$ is the exact field and $G_+(\cdot)$ is an independent continuous Gaussian field with covariance $f_+(\cdot, \cdot)$, and
\begin{align*}
\EE\left[\mu_{f, g}(B_1)^h \mu_{f, g}(B_2)^h \right]
& \le ||g||_{\infty}^{2h} \EE\left[\mu_{f}(B_1)^h \mu_{f}(B_2)^h \right]\\
& \le \frac{||g||_{\infty}^{2h}}{\EE \left[e^{2h \min_{x \in \overline{D}} (\sqrt{2d}G_-(x) - d \EE[G_-(x)^2])}\right]} \EE\left[\mu_{f_+}(B_1)^h \mu_{f_+}(B_2)^h \right]\\
& \le \frac{||g||_{\infty}^{2h}\EE \left[e^{2h \max_{x \in \overline{D}} (\sqrt{2d}G_+(x) - d \EE[G_+(x)^2])}\right]}{\EE \left[e^{2h \min_{x \in \overline{D}} (\sqrt{2d}G_-(x) - d \EE[G_-(x)^2])}\right]}  \EE\left[\mu_{0}(B_1)^h \mu_{0}(B_2)^h \right]
\end{align*}

\noindent where $\mu_0$ is the critical GMC associated to $Y_0$, and the factor in front of the cross moment on the last line is finite by \Cref{lem:ctsGP}.

We start with the simple case where $\mathrm{dist}(B_1, B_2) > 0$.
\begin{lem}\label{lem:cross_dis}
Let $A_1, A_2 \subset B(0, r_d)$ be two Borel sets such that $\delta :=\mathrm{dist}(A_1, A_2) > 0$. Then for any $h \in [0, 1)$
\begin{align*}
\EE\left[\mu_{0}(A_1)^h \mu_{0}(A_2)^h \right]
\le C_h \delta^{-dh}.
\end{align*}

\noindent for some constant $C_h > 0$ independent of $A_1$ and $A_2$.
\end{lem}

\begin{proof}
Recall that the random variable $\mu_{0}(A_1) \mu_0(A_2)$ is the limit of
\begin{align*}
\left(\log \frac{1}{\epsilon}\right)\int_{A_1} \int_{A_2} e^{\sqrt{2d}\left(Y_{0, \epsilon}(x) + Y_{0, \epsilon}(y)\right) - d \left( \EE[Y_{0, \epsilon}(x)^2]+\EE[Y_{0, \epsilon}(x)^2]\right)}dxdy
\end{align*}

\noindent as $\epsilon \to 0^+$ where $Y_{0, \epsilon} = Y_0 \ast \theta_\epsilon$ for some smooth mollifier $\theta$ in the construction of critical GMCs. If we write $\mathbf{Y}_{\epsilon}(x, y) = Y_{0, \epsilon}(x) + Y_{0, \epsilon}(y) + N_K$ where $N_K$ is an independent $N(0, K)$ variable, then
\begin{align}
\notag & \EE\left[\mu_{0, \epsilon}(A_1)^h \mu_{0, \epsilon}(A_2)^h\right]
= \frac{1}{\EE\left[ e^{h \left(\sqrt{2d}N_k - d \EE[N_K^2] \right)}\right]} \EE\left[ e^{h \left(\sqrt{2d}N_k - d \EE[N_K^2] \right)}\mu_{0, \epsilon}(A_1)^h \mu_{0, \epsilon}(A_2)^h\right]\\
\label{eq:cross_cmp}
& =  C_K \left(\log \frac{1}{\epsilon}\right)^h \EE\left[\left(\int_{A_1} \int_{A_2}   e^{d\EE[Y_{0, \epsilon}(x) Y_{0, \epsilon}(y)]} e^{\sqrt{2d}\left(\mathbf{Y}_{\epsilon}(x, y)\right) - d\EE[\mathbf{Y}_{\epsilon}(x, y)^2]} dxdy\right)^h\right].
\end{align}

\noindent and the term $e^{d\EE[Y_{0, \epsilon}(x) Y_{0, \epsilon}(y)]}$ in the integral is bounded above uniformly in $x, y, \epsilon$ under the assumption that $\delta > 0$. We then introduce a new field $\widetilde{\mathbf{Y}}_{\epsilon}(x, y) = Y_{0, \epsilon}(x) +  \widetilde{Y}_{0, \epsilon}(y)$ where $\widetilde{Y}_{0, \epsilon}$ is an independent copy of $Y_{0, \epsilon}$. For $K>0$ sufficiently large, one can check that
\begin{align*}
\EE\left[ \mathbf{Y}_{\epsilon}(x_1, y_1) \mathbf{Y}_{\epsilon}(x_2, y_2)\right]
\ge \EE\left[ \widetilde{\mathbf{Y}}_{\epsilon}(x_1, y_1) \widetilde{\mathbf{Y}}_{\epsilon}(x_2, y_2)\right]
\end{align*}

\noindent Using \Cref{lem:Kahane}, we see that \eqref{eq:cross_cmp} is bounded by
\begin{align*}
& C_K e^{dh \sup_{(x, y) \in (A_1, A_2)} \EE[Y_{0, \epsilon}(x) Y_{0, \epsilon}(y)]}  \left(\log \frac{1}{\epsilon}\right)^h \EE\left[\left(\int_{A_1} \int_{A_2} e^{\sqrt{2d}\left(\widetilde{\mathbf{Y}}_{\epsilon}(x, y)\right) - d\EE[\widetilde{\mathbf{Y}}_{\epsilon}(x, y)^2]} dxdy\right)^h\right]\\
& \qquad \le C_K' \delta^{-dh} \EE\left[\mu_{0, \epsilon}(A_1)^h\right] \EE\left[ \mu_{0, \epsilon}(A_2)^h\right]
\end{align*}

\noindent and the above line remains finite as we pass it to the limit $\epsilon \to 0^+$, i.e.
\begin{align*}
\EE\left[\mu_{0}(A_1)^h \mu_{0}(A_2)^h \right] 
& \le C_K' \delta^{-dh} \liminf_{\epsilon \to 0^+} \EE\left[\mu_{0, \epsilon }(A_1)^h\right] \EE\left[ \mu_{0, \epsilon}(A_2)^h\right]
\le C_K'' \delta^{-dh}.
\end{align*}
\end{proof}

\begin{proof}[Proof of \Cref{lem:cross}]
It suffices to consider the cases where
\begin{align*}
A_1 = A_1^k(r) = [0, r]^d \qquad \text{and} \qquad A_2 = A_2^k(r) = [-r, 0]^{d-k} \times [0, r]^{k}
\end{align*}

\noindent for some fixed $4r \in (0, r_d)$ and each $k \in \{0, 1, \dots, d-1\}$. We shall call pairs of $d$-cubes of the form $(A_1^k(r), A_2^k(r))$ $k$-configurations of size $r$, and prove that 
\begin{align*}
\Ca_k(r) := \EE\left[\mu_{0}(A_1^k(r))^h \mu_{0}(A_2^k(r))^h \right] < \infty
\end{align*}

\noindent for every such $k$ by induction.

When $k = 0$, observe that $A_1^0(r) \cap A_2^0(r) = \{0\}$. we decompose $A_{i}^0(r)$ as follows:
\begin{align*}
A_1^0(r) &= [0, r/2]^d \cup \left([0,r]^d \setminus [0, r/2]^d\right)\\
&=: A_1^0(r 2^{-1}) \cup C_1^0(r),\\
A_2^0(r) & = [-r/2, 0]^d \cup \left([-r, 0]^d \setminus [-r/2, 0]^d\right)\\
&=: A_2^k(r 2^{-1}) \cup C_2^0(r).
\end{align*}

\noindent By the concavity of $x \mapsto x^h$ for $h \in (0, 1)$, we see that
\begin{align*}
& \EE\left[\mu_{0}(A_1^0(r))^h \mu_{0}(A_2^0(r))^h \right]
\le\EE\left[\mu_{0}(A_1^0(r2^{-1}))^h \mu_{0}(A_2^0(r2^{-1}))^h \right]\\
&\qquad  +\EE\left[\mu_{0}(A_1^0(r2^{-1}))^h \mu_{0}(C_2^0(r))^h \right]
+ \EE\left[\mu_{0}(A_2^0(r2^{-1}))^h \mu_{0}(C_1^0(r))^h \right]
+ \EE\left[\mu_{0}(C_1^0(r))^h \mu_{0}(C_2^0(r))^h \right]\\
& \qquad =: \EE\left[\mu_{0}(A_1^0(r2^{-1}))^h \mu_{0}(A_2^0(r2^{-1}))^h \right] + \Ra_{k=0}(r),
\end{align*}

\noindent i.e. we end up with an expectation associated to a $0$-configuration of size $r/2$ as well as residual terms $\Ra_{0} := \Ra_0(r)$ which is finite by \Cref{lem:cross_dis} since the cross moments only involve pairs of sets that are away from each other by at least $r/2$. We can iterate the above decomposition for $(A_1^0(r2^{-1}), A_2^0(r2^{-1}))$ so that one obtains $A_i^0(r2^{-n}) = A_i^0(r2^{-n-1}) \cup C_i^0(r2^{-n})$ for each $n \in \NN$, and obtain 
\begin{align}
\notag \Ca_0(r)
& \le \sum_{n \ge 0} \Big( \EE\left[\mu_{0}(A_1^0(r2^{-(n+1)}))^h \mu_{0}(C_2^0(r2^{-n}))^h \right]
+  \EE\left[\mu_{0}(A_2^0(r2^{-(n+1)}))^h \mu_{0}(C_1^0(r2^{-n}))^h \right]\\
\notag & \qquad +\EE\left[\mu_{0}(C_1^0(r2^{-n}))^h \mu_{0}(C_2^0(r2^{-n}))^h \right]\Big)\\
\label{eq:cross0_int} &=: \sum_{n \ge 0}\Ra_{0}(r2^{-n}).
\end{align}

\noindent Note that the sets involved in $\Ra_0(r2^{-n})$ are essentially those in $\Ra_0$ up to a scaling factor of $2^{-n}$. We then appeal to the scaling property of the exact field: if $S_1, S_2 \subset B(0, r_d)$ and $c \in (0,1)$, then
\begin{align*}
\mu(cS_1) \mu(cS_2) 
& \overset{d}{=} \left(c^d e^{\sqrt{2d}N_c - d \EE[N_c^2]} \mu_d(S_1)\right)\left(c^d e^{\sqrt{2d}N_c - d \EE[N_c^2]} \mu_d(S_2)\right)\\
& = \left(e^{\sqrt{2d}N_c - 2d \EE[N_c^2]}\right)^2 \mu_d(S_1)\mu_d(S_2)
\end{align*}

\noindent where $N_c \sim N(0, -\log c)$ is independent of $\mu_d$. This suggests that
\begin{align*}
\Ra_0(r2^{-n}) 
& = \Ra_0 \times \EE \left[ \left(e^{\sqrt{2d}N_{2^{-n}} - 2d \EE[N_{2^{-n}}^2]} \right)^{2h} \right]
= \Ra_0 2^{n (4dh^2 - 4dh)} 
= \Ra_0 2^{n \left[4d(h - \frac{1}{2})^2  - d\right]} 
\end{align*}

\noindent which is summable for $h \in (\frac{1}{2}, \frac{1}{2} + \frac{1}{2\sqrt{d}})$.

Now consider $k \ge 1$. We decompose our cubes $A_i^k$ as follows:
\begin{align*}
A_1^k(r) &= [0, r/2]^{d-k} \times [0, r]^k \cup \left([0,r]^d \setminus [0, r/2]^{d-k} \times [0, r]^k\right)\\
&=: \widehat{A}_1^k(r) \cup C_1^k(r),\\
A_2^k(r) & = [-r/2, 0]^{d-k} \times [0, r]^k \cup \left([-r, 0]^d \setminus [-r/2, 0]^{d-k} \times [0, r]^k\right)\\
&=: \widehat{A}_2^k(r) \cup C_2^k(r).
\end{align*}

\noindent This means that
\begin{align*}
& \EE \left[ \mu_0(A_1^k(r))^h \mu_0(A_2^k(r))^h \right]
\le \EE \left[ \mu_0(\widehat{A}_1^k(r))^h \mu_0(\widehat{A}_2^k(r))^h \right]\\
& \qquad + \EE \left[ \mu_0(\widehat{A}_1^k(r))^h \mu_0(C_2^k(r))^h \right]
+ \EE \left[ \mu_0(\widehat{A}_2^k(r))^h \mu_0(C_1^k(r))^h \right]
+ \EE \left[ \mu_0(C_1^k(r))^h \mu_0(C_2^k(r))^h \right]\\
& \qquad =: \EE \left[ \mu_0(\widehat{A}_1^k(r))^h \mu_0(\widehat{A}_2^k(r))^h \right] + \Ra_k(r).
\end{align*}

\noindent Again we have $\Ra_k := \Ra_k(r) < \infty$ by \Cref{lem:cross_dis} because it only involves cross moments associated to pairs of sets separated by a distance of at least $r/2$. 

Now in order to use the scaling property of $\mu_d$, we have to further decompose $\widehat{A}_i^k(r)$ into $d$-cubes of length $r/2$, i.e. we have to do further partitioning\footnote{Our $d$-cubes will always be closed and so this is not a partition strictly speaking, but only the boundary may be double-counted at most countably many times which is negligible with respect to the Lebesgue/GMC measure.} with respect to each of the last $k$ coordinates depending on whether it lies in $[0, r/2]$ or $[r/2, 0]$. Pick two such sub-cubes, one from the partitioning of $\widehat{A}_1^k(r)$ and the other from that of $\widehat{A}_2^k(r)$. The pair gives a $j$-configuration of size $r/2$, where $j \in \{0, 1, \dots, k\}$ depending on how many sides the cubes share with each other, and there are exactly $2^k \binom{k}{j}$ $j$-configurations. Therefore, 
\begin{align*}
\Ca_k(r) = \EE \left[ \mu_0(A_1^k(r))^h \mu_0(A_2^k(r))^h \right]
& \le \Ra_k(r) + 2^k \sum_{j=0}^k \binom{k}{j} \Ca_j(r2^{-1})\\
& \le 2^k \Ca_k(r2^{-1}) + \Ra_k(r) + 2^k \sum_{j=0}^{k-1} \binom{k}{j} \Ca_j(r2^{-1}).
\end{align*}

\noindent (Recall that $\Ca_j(r2^{-1})$ is the cross moment associated to a $j$-configuration of size $r/2$.) The decomposition can be repeated and we obtain the bound
\begin{align*}
\Ca_k(r)
& \le \sum_{n \ge 0} 2^{nk} \left(\Ra_k(r2^{-n}) + 2^k \sum_{j=0}^{k-1} \binom{k}{j} \Ca_j(r2^{-(n+1)})\right)\\
& = \left(\Ra_k(r) + 2^k \sum_{j=0}^{k-1} \binom{k}{j} \Ca_j(r2^{-1})\right) \sum_{n \ge 0} 2^{nk} 2^{n \left[4d(h - \frac{1}{2})^2  - d\right]} 
\end{align*}

\noindent from the same scaling consideration. By induction hypothesis, the multiplicative factor in front of the sum is finite, while the summand $2^{n \left[4d(h - \frac{1}{2})^2  - (d-k)\right]}$ is summable provided that $h \in (\frac{1}{2}, \frac{1}{2} + \frac{1}{2\sqrt{d}})$ since $k \le d-1$. This concludes the proof.
\end{proof}

\section{Reference log-correlated Gaussian field}\label{app:ref}
For the purpose of computing the constant $\overline{C}_d$ in \Cref{subsec:eval} we introduced the reference field $X_d(\cdot)$ for $d \ge 2$. This appendix is devoted to the construction of this Gaussian field and the existence of the corresponding critical GMC.

\subsection{Construction of reference field $X_d$}
Let $L>0$ be such that $r_d(L) > 1$, and consider the $L$-exact field $Y_L(\cdot)$ on $B(0, 1)$. Inspired by the $d=2$ case we consider spherical averages of $Y_L$.

\begin{lem}\label{lem:Xdcont}
For $0 \ne x \in B(0, 1)$, we define
\begin{align*}
\overline{Y}_L(x) = \int Y_L(u) \sigma_{|x|}(du)
\end{align*}

\noindent where $\sigma_s(du)$ is the uniform measure on the $(d-1)$-sphere $\mathbb{S}^{d-1}(s)$ of radius $s$. Then the centred Gaussian process $\overline{Y}_L(\cdot)$ satisfies
\begin{align}\label{eq:Ybar}
& \EE \left[\overline{Y}_L(x) \overline{Y}_L(y)\right] = - \log |x| \vee |y|  + L + S_d(x, y)
\end{align}

\noindent where
\begin{align*}
S_d(x, y) = -\frac{|\mathbb{S}^{d-2}|}{2|\mathbb{S}^{d-1}|} \int_{-1}^1 (1-u^2)^{\frac{d-2}{2}}\log |1 -2cu + c^2|  du,
\qquad c = \frac{|x|}{|y|} \wedge \frac{|y|}{|x|}.
\end{align*}

\noindent In particular, the process $(\overline{Y}_L(e^{-t}))_{t \ge 0}$ is continuous almost surely.
\end{lem}

\begin{proof}
By definition,
\begin{align*}
\EE \left[\overline{Y}_L(x) \overline{Y}_L(y)\right]
& = \int \int \EE\left[Y_L(u) Y_L(v)\right] \sigma_{|x|}(du) \sigma_{|y|}(dv)\\
& = L - \int \int \log |u-v| \sigma_{|x|}(du) \sigma_{|y|}(dv).
\end{align*}

Suppose $|x| \ge |y|$. By rotational invariance, it is not difficult to see that
\begin{align*}
\int \int \log |u-v| \sigma_{|x|}(du) \sigma_{|y|}(dv)
& = \int \log ||x|u-|y|e_1| \sigma_{1}(du)\\
& = \log |x| + \int \log |u-ce_1| \sigma_{1}(du), \qquad c = \frac{|y|}{|x|}.
\end{align*}

\noindent where $e_1 = (1, 0, \dots, 0)^T$. But then
\begin{align}
\notag 
\int \log |u-ce_1| \sigma_{1}(du)
& = \frac{1}{2} \int \log \left| (u_1 - c)^2 + \sum_{i \ge 2} u_i^2 \right| \sigma_1(du)\\
\label{eq:Sdcal}
& = \frac{1}{2|\mathbb{S}^{d-1}|} \int_{-1}^1 \int \log |1 - 2cu_1 + c^2|\tilde{\sigma}_{\sqrt{1-u_1^2}}(du_{-1}) du_1
\end{align}

\noindent where $\tilde{\sigma}_{\sqrt{1-u_1^2}}(du_{-1})$ is the area measure on the $(d-2)$-sphere of radius $\sqrt{1-u_1^2}$ and with coordinates $(u_2, u_3, \dots, u_d)$. This shows that \eqref{eq:Sdcal} is equal to $-S_d(x, y)$, and we have verified the covariance formula \eqref{eq:Ybar}.

We now proceed to the claim of continuity, and consider only $d \ge 3$ since it is well-known that $S_2(x, y) \equiv 0$ in which case $(\overline{Y}_L(e^{-t})_{t \ge 0}$ is a Brownian motion (starting from an independent random position $\sim N(0, L)$).

Let $t \ge s \ge 0$. We have
\begin{align*}
\EE \left[\left(\overline{Y}_L(e^{-t}) - \overline{Y}_L(e^{-s})\right)^2\right]
& = \EE \left[\overline{Y}_L(e^{-t})^2\right] + \EE \left[\overline{Y}_L(e^{-s})^2\right] - 2\EE \left[\overline{Y}_L(e^{-t})\overline{Y}_L(e^{-s})\right]\\
& = (t - s) + S_d(e^{-t}, e^{-t}) + S_d(e^{-s}, e^{-s}) - 2S_d(e^{-t}, e^{-s})
\end{align*}

\noindent where
\begin{align*}
& S_d(e^{-t}, e^{-t}) + S_d(e^{-s}, e^{-s}) - 2S_d(e^{-t}, e^{-s})
= 2\left|S_d(1, 1) - 2S_d(e^{-(t-s)}, 1) \right|\\
& \qquad= \frac{|\mathbb{S}^{d-2}|}{|\mathbb{S}^{d-1}|}  \left|\int_{e^{-(t-s)}}^1 \int_{-1}^1 (1-u^2)^{\frac{d-2}{2}}\frac{2(c-u)}{1 -2cu + c^2} du dc\right|\\
& \qquad\le  \frac{|\mathbb{S}^{d-2}|}{|\mathbb{S}^{d-1}|}  \int_{e^{-(t-s)}}^1 \int_{-1}^1 (1-u^2)^{\frac{d-2}{2}}\frac{|2(c-u)|}{(c-u)^2 + 1-u^2} du dc\\
& \qquad \le 4 \frac{|\mathbb{S}^{d-2}|}{|\mathbb{S}^{d-1}|} \int_{e^{-(t-s)}}^1 \int_{-1}^1 (1-u^2)^{-\frac{1}{2}}du dc
\le C (1 - e^{-(t-s)})
\le C(t-s).
\end{align*}

\noindent Therefore $\EE \left[\left(\overline{Y}_L(e^{-t}) - \overline{Y}_L(e^{-s})\right)^2\right] \le (C+1)(t-s)$, and by \Cref{lem:contGP} the Gaussian process $(\overline{Y}_L(e^{-t}))_{t \ge 0}$ is  H\"older continuous.

\end{proof}

We now continue with the construction of our reference field. By a straightforward covariance computation (using rotational symmetry), we see that the Gaussian field
\begin{align}\label{eq:Yhat}
\widehat{Y}_d(\cdot) := Y_L(\cdot) - \overline{Y}_L(\cdot)
\end{align} 

\noindent is independent of $\overline{Y}_L(\cdot)$ and satisfies
\begin{align*}
\EE\left[\widehat{Y}_d(x)\widehat{Y}_d(y)\right] = \log \frac{|x| \vee |y|}{|x-y|} - S_d(x, y)
\end{align*}

\noindent which is scale invariant. In particular $\widehat{Y}_d(\cdot)$ does not depend on $L$ and may be defined on Euclidean balls of arbitrary size.

If $(B_t)_{t \ge 0}$ is an independent standard Brownian motion, we can write $\overline{X}(x) := B_{-\log |x|}$ and define our reference field by $X_d(\cdot) = \overline{X}(x) + \widehat{Y}_d(\cdot)$ which is a centred Gaussian field on the unit ball with
\begin{align*}
\EE[X_d(x) X_d(y)]
&= \EE\left[\overline{X}(x) \overline{X}(y)\right] +  \EE\left[\widehat{Y}_d(x) \widehat{Y}_d(y)\right]
= - \log |x-y| - S_d(x, y).
\end{align*}

\subsection{Existence of associated critical GMC $\mu_d$} \label{app:existsGMC}
We would like to argue that the sequence of measures
\begin{align*}
\mu_{d, \epsilon}(dx) := \left(\log \frac{1}{\epsilon}\right)^{1/2} e^{\sqrt{2d}X_{d, \epsilon}(x) - d \EE[X_{d, \epsilon}(x)^2]}dx,
\end{align*}

\noindent where $X_{d, \epsilon}(x) = X_d \ast \theta_{\epsilon}(x)$, converges in probability to some measure $\mu_d$ in the space of Radon measures equipped with the weak$^*$ topology. To do so, we first show that the claim of convergence is true  on any subset $D_n = D \setminus B(0, \kappa_n)$ for some sequence of $\kappa_n > 0$ tending to $0$ as $n \to \infty$. Pick $L > 0$ sufficiently large such that $r_d(L) > 1$. Using the construction of our reference field, we have
\begin{align*}
X_d(\cdot) = Y_L(\cdot) - \overline{Y}_L(\cdot) + \overline{X}(\cdot) \qquad a.s.
\end{align*}

\noindent and so
\begin{align*}
\mu_{d, \epsilon}(dx) 
& = e^{\sqrt{2d} (\overline{X}_\epsilon(x) - \overline{Y}_{L, \epsilon}(x))}
e^{- d\EE[\overline{X}_{\epsilon}(x)^2 + \widehat{Y}_{d, \epsilon}(x)^2 - Y_{L, \epsilon}(x)^2]} 
\left(\log \frac{1}{\epsilon}\right)^{1/2} e^{\sqrt{2d} Y_{L, \epsilon}(x) - d \EE[Y_{L, \epsilon}(x)^2]}dx\\
&= e^{\sqrt{2d} (\overline{X}_\epsilon(x) - \overline{Y}_{L, \epsilon}(x))}
e^{- d\EE[\overline{X}_{\epsilon}(x)^2 + \widehat{Y}_{d, \epsilon}(x)^2 - Y_{L, \epsilon}(x)^2]} 
\mu_{L, \epsilon}(dx).
\end{align*}

\noindent As $\epsilon \to 0^+$, $\mu_{L, \epsilon}$ converges in probability to the critical GMC $\mu_L$ associated to the $L$-exact field $Y_L$ on $D$. If we restrict ourselves to $\overline{D}_n$, we see that $\overline{X}(\cdot)$ and $\overline{Y}_L(\cdot)$ are H\"older continuous functions, and so $e^{\sqrt{2d} (\overline{X}_\epsilon(x) - \overline{Y}_{L, \epsilon}(x))}$ converges uniformly to $e^{\sqrt{2d} (\overline{X}(x) - \overline{Y}_L(x))}$ almost surely. Also, it is easy to check that
\begin{align*}
\EE[\overline{X}_{\epsilon}(x)^2 + \widehat{Y}_{d, \epsilon}(x)^2 - Y_{L, \epsilon}(x)^2]
= -\int \int \theta(u) \theta(v) S_d(x+\epsilon u, x+\epsilon v) du dv
\end{align*}

\noindent where $S_d(x, y)$ is also a H\"older continuous on $\overline{D}_n$ by \Cref{lem:Xdcont}, and so $e^{- d\EE[\overline{X}_{\epsilon}(x)^2 + \widehat{Y}_{d, \epsilon}(x)^2 - Y_{L, \epsilon}(x)^2]}$ converges uniformly to $e^{d S_d(x, x)}$. Combining all these considerations, we see that $\mu_{d, \epsilon}(dx)$ converges to a critical GMC $\mu_d^n(dx)$ on $D_n$ in probability as $\epsilon \to 0^+$.

Next, we extend the definition of each $\mu_d^n$ to the entire domain $D$ by defining $\mu_d^n(B(0, \kappa_n)) = 0$. This gives us a non-decreasing sequence of measures $\mu_d^1 \le \mu_d^2 \le ..$ on $D$ with the property that $\mu_d^n|_{D_m} = \mu_d^m$ for any $m \le n$. We argue that the sequence of measures $\{\mu_d^i\}_{i \ge 1}$ is tight: if we cover $D$ by finitely many balls $A_i$ of radius at most $ r_d(-L)$, then for $q \in (0, 1)$
\begin{align*}
\EE \left[ \left(\lim_{n \to \infty} \mu_d^n(D)\right)^q \right]
\le \sum_i \EE \left[\left(\lim_{n \to \infty} \mu_d^n(A_i)\right)^q \right]
& = \sum_{i} \lim_{n \to \infty} \EE \left[\mu_d^n(A_i)^q \right]\\
& \le \sum_{i} \EE \left[ \mu_{-L}(B(0, r_d(-L)))^q \right] < \infty
\end{align*}

\noindent where we have used Gaussian comparison \eqref{eq:Gcomp} in the second inequality (assuming that $L > ||S_d||_{\infty}$), i.e. the total mass of $\mu_d^n(D)$ is uniformly bounded in $n$ almost surely. From this we obtain by Prokhorov's theorem that $\{\mu_d^i\}_i$ is relatively compact and so there exists a subsequence $\{n_k\}_k$ along which $\mu_d^{n_k}$ converges weakly to some random measure $\mu_d$ almost surely. However, by splitting any $g \in C_b(D)$ into positive and negative parts $g = g_+ - g_-$, we see that
\begin{align*}
\mu_d(g) = \mu_d(g_+) - \mu_d(g_-) 
= \lim_{n \to \infty} \left[\mu_d^n(g_+) - \mu_d^n(g_-)\right]
= \lim_{n \to \infty} \mu_d^n(g)
\end{align*}

\noindent by monotone convergence, i.e. $\mu_d^n$ converges weakly to $\mu_d$ almost surely.

Finally we show that $\mu_d$ is indeed the weak$^*$ limit of $\mu_{d, \epsilon}$ on the whole domain $D$ in probability, and by standard argument (see e.g. \cite[Chapter 4]{Kal2017}) it suffices to check that for any fixed $g \in C_b(D)$,
\begin{align*}
\mu_{d, \epsilon}(g) \xrightarrow{\epsilon \to 0^+} \mu_d(g) \qquad \text{in probability}.
\end{align*}

\noindent For this, consider for any $\delta > 0$
\begin{align*}
\PP\left( |\mu_{d, \epsilon}(g) - \mu_{d}(g)| > \delta\right)
& \le \PP\left( |\mu_{d, \epsilon}(g1_{D_n}) - \mu^n_{d}(g)| > \frac{\delta}{3}\right) \\
& \qquad + \PP\left( ||g||_\infty \mu_{d, \epsilon}(D \setminus D_n) > \frac{\delta}{3} \right)
+ \PP\left( ||g||_\infty \mu_{d}(D \setminus D_n) > \frac{\delta}{3} \right).
\end{align*}

\noindent The first term on the RHS converges to $0$ as $\epsilon \to 0^+$ as $\mu_{d, \epsilon}$ converges in probability to $\mu_d^n$ on $D_n$. For the second term, we have for $q \in (0, 1)$
\begin{align*}
\PP\left( ||g||_\infty\mu_{d, \epsilon}(D \setminus D_n) > \frac{\delta}{3} \right)
& \le \left(\frac{3||g||_\infty}{\delta}\right)^q \EE \left[ \mu_{d, \epsilon}(B(0, \kappa_n))^q \right]\\
& \le \left(\frac{3||g||_\infty}{\delta}\right)^q \EE \left[ \mu_{-L, \epsilon}(B(0, \kappa_n))^q \right]
\end{align*}

\noindent where the last inequality again follows from Gaussian comparison. The third term may be bounded similarly and overall we have
\begin{align*}
\limsup_{\epsilon \to 0^+} \PP\left( |\mu_{d, \epsilon}(g) - \mu_{d}(g)| > \delta\right)
\le 2 \left(\frac{3||g||_\infty}{\delta}\right)^q \limsup_{\epsilon \to 0^+} \EE \left[ \mu_{-L, \epsilon}(B(0, \kappa_n))^q \right].
\end{align*}

\noindent Since $n$ is arbitrary, we let $n \to \infty$ or equivalently $\kappa_n \to 0^+$ to make the above bound arbitrarily small (by \Cref{lem:GMC_moment}(ii)) and conclude that $\mu_{d, \epsilon}$ converges in probability to $\mu_d$ on $D$ in the weak$^*$ topology.

\section{Fusion estimates} \label{app:fusion}
This appendix is devoted to \Cref{lem:fusion}. We first explain the main ideas that are inspired by earlier works \cite{DKRV2017, BW2018} on fusion estimates of GMCs, and then sketch the technical estimates and arguments for the proof of \Cref{lem:fusion}.

\subsection{Main idea: exponential functional of Brownian motion}\label{subsec:fusion_idea}
By construction, $X_d$ has the decomposition
\begin{align*}
X_{d}(x) = \overline{X}(x) + \widehat{Y}_d(x)
\end{align*}

\noindent where $(B_t)_{t \ge 0} = \overline{X}(e^{-t})_{t \ge 0}$ is a standard Brownian motion. This is also translated into
\begin{align}\label{eq:ref_dec}
X_{d, \epsilon}(x)
= \overline{X}_{\epsilon} (x) + \widehat{Y}_{d, \epsilon}(x)
\qquad \text{and} \qquad B_{\epsilon, t} := \overline{X}_{\epsilon}(e^{-t})
\end{align}

\noindent where $\overline{X}_{\epsilon}$ and $\widehat{Y}_{d, \epsilon}$ are defined analogously.

Now recall the definition of $\mu_{d, \epsilon}(v, A)$ from \eqref{eq:lim_fusion2}. We shall split this random variable into three terms:
\begin{align*}
\mu_{d, \epsilon}(v, r)
& = \left(\log \frac{1}{\epsilon}\right)^{1/2} \int_{\{2\epsilon \le |x-v| \le r\} \cap A} \frac{e^{\sqrt{2d} X_{d, \epsilon}(x-v) - d \EE\left[X_{d, \epsilon}(x-v)^2\right]}}{\left(|x-v|\vee \epsilon\right)^{2d}} dx\\
& = e^{\sqrt{2d} B_{\epsilon, -\log r}}\int_{-\log r}^{-\log 2\epsilon} e^{\sqrt{2d} (B_{\epsilon, s}-B_{\epsilon, -\log r})}Z_{d, \epsilon}^{A, v}(ds),\\
\mu_{d, \epsilon}^c(v)
& = \left(\log \frac{1}{\epsilon}\right)^{1/2} \int_{\{|x-v| \le 2\epsilon\} \cap A} \frac{e^{\sqrt{2d} X_{d, \epsilon}(x-v) - d \EE\left[X_{d, \epsilon}(x-v)^2\right]}}{\left(|x-v|\vee \epsilon\right)^{2d}} dx\\
& \le \left(\log \frac{1}{\epsilon}\right)^{1/2} \epsilon^{-2d} \int_{\{|x-v| \le 2\epsilon\} \cap A} e^{\sqrt{2d} X_{d, \epsilon}(x-v) - d \EE\left[X_{d, \epsilon}(x-v)^2\right]} dx\\
& = 2^{2d} e^{\sqrt{2d} B_{\epsilon, -\log 2\epsilon}} \underbrace{\int_{-\log 2\epsilon}^{\infty} e^{\sqrt{2d} [(B_{\epsilon, s} - B_{\epsilon, -\log 2\epsilon}) - (s + \log 2\epsilon)]}Z_{d, \epsilon}^{A, v}(ds)}_{=:\overline{\mu}_{d, \epsilon}^c(v)},\\
R_{d, \epsilon}(v, r) 
& = \left(\log \frac{1}{\epsilon}\right)^{1/2} \int_{\{ |x-v| > r\} \cap A} \frac{e^{\sqrt{2d} X_{d, \epsilon}(x-v) - d \EE\left[X_{d, \epsilon}(x-v)^2\right]}}{\left(|x-v|\vee \epsilon\right)^{2d}} dx,\\
\end{align*}

\noindent where
\begin{align*}
Z_{d, \epsilon}^{A, v}(ds)
= \left(\log \frac{1}{\epsilon}\right)^{1/2} \int_{\mathbb{S}^{d-1}} 1_A(v + e^{-s}x) e^{\sqrt{2d} \widehat{Y}_{d, \epsilon}(e^{-s} u) - d\EE\left[\widehat{Y}_{d, \epsilon}(e^{-s}u)^2\right]} \sigma_1(du) ds.
\end{align*}

\noindent As $\epsilon \to 0^+$, we see that
\begin{itemize}
\item $R_{d, \epsilon}(v, r)$ converges to some finite random variables $R_{d}(v, r)$ by the construction of the critical GMC $\mu_d$ associated to our reference field $X_d$;
\item $Z_{d, \epsilon}^{A, v}(ds)$ converges to some $Z_{d}^{A,v}(ds)$ which is the critical GMC associated to $\widehat{Y}_d$ (which exists based on arguments similar to that in \Cref{app:existsGMC}), expressed in terms of spherical coordinates with all the angular coordinates marginalised out;
\item $\overline{\mu}_{d, \epsilon}^c(v, r)$, which appears in the definition of $\mu_{d, \epsilon}^c(v)$, is essentially the mass of $\{|x-v| \le 2\epsilon \}$ normalised to order $1$ (by taking out all the extra factors after applying the substitution $\epsilon u = x-v$ and the scale invariance of $\widehat{Y}_d(\cdot)$).
\end{itemize}

Therefore, to get the idea of how $\EE\left[ e^{-\lambda \mu_{d, \epsilon}(v, A)}\right]$ behaves as $\epsilon \to 0^+$, we may consider the toy model
\begin{align}\label{eq:toy}
\EE\left[ e^{-\lambda (U_t+ V_t + W)}\right]
\end{align}

\noindent with $U_t = \int_0^t e^{\sqrt{2d} B_s} ds$, $V_t = e^{\sqrt{2d}B_t} V$ where $(V, W)$ are some finite independent random variables, $(B_t)_{t \ge 0}$ is a standard Brownian motion, and $t = -\log 2\epsilon \to \infty$. The tuple $(U_t, V_t, W)$ should be seen as the toy version of $\left(\mu_{d, \epsilon}(v, r), \overline{\mu}_{d, \epsilon}^c(v), R_{d, \epsilon}(v, r)\right)$.

Similar to the observation in \cite{DKRV2017, BW2018}, it happens that the leading order contribution to \eqref{eq:toy} as $t \to \infty$ comes from the event
\begin{align*}
\left\{ \sup_{s \le t} B_s = O(1)\right\},
\end{align*}

\noindent the probability of which is of order $1 / \sqrt{t}$, explaining the renormalisation factor $\left(\log 1/\epsilon \right)^{1/2}$ appearing on the LHS of \eqref{eq:fusion}. On this event, it is not difficult to check that the terminal value $B_t$ of the Brownian motion is extremely negative (it is less than $-t^{\frac{1}{2}-}$ with high probability) and so $V_t$ vanishes in the limit.

For $U_t$ we need finer description of the Brownian path. When $\sup_{s \le t} B_s = x \in \RR_+$, it happens that the behaviour of $(B_s)_{s \le t}$ is very similar to the following:
\begin{itemize}
\item First, it evolves like a standard Brownian motion until $s = T_{x} = \inf\{u > 0: B_u = x\}$ (which is $o(t)$ with high probability) when it reaches its maximum value.
\item It then evolves like $B_s = x - \beta_{s - T_x}$, where $(\beta_s)_{s \ge 0}$ is an independent $\mathrm{BES}_0(3)$-process (hence explaining why $B_t$ is extremely negative).
\end{itemize}

\noindent If one further applies \Cref{lem:time_rev}, one sees that
\begin{align*}
U_t
& = \int_0^t e^{\sqrt{2d}B_s}ds
\approx e^{\sqrt{2d}x} \left[\int_0^{T_x} e^{-\sqrt{2d}(B_s - x)} ds + \int_{T_x}^t e^{-\sqrt{2d}\beta_{s - T_x}}ds\right]\\
& \xrightarrow[t \to \infty]{d} e^{\sqrt{2d}x} \int_{-L_{x, -}}^\infty e^{-\sqrt{2d}\beta_s} ds
\end{align*}

\noindent where $(\beta_{-s})_{s \ge 0}$ is an independent $\mathrm{BES}_0(3)$-process and $L_{x, -} := \sup \{s > 0: \beta_{-s} = x\}$. Finally, the maximum value $x$ attained by the Brownian motion is asymptotically ``uniformly" distributed, and we obtain
\begin{align*}
\EE\left[ e^{-\lambda (U_t+ V_t + W)}\right]
\overset{t \to \infty}{\sim}
\sqrt{\frac{2}{\pi t}} \int_0^\infty \EE \left[ \exp \left(- \lambda\left( e^{\sqrt{2d}x} \int_{-L_{x, -}}^\infty e^{-\sqrt{2d}\beta_s} ds + W \right)\right)\right] dx
\end{align*}

\noindent which is of the same form as \eqref{eq:fusion}.

\subsection{Some estimates}
Let us collect a few estimates that will be used in the proof of \Cref{lem:fusion}.

Suppose $X_{d, \epsilon}(\cdot) := X_{d} \ast \theta_{\epsilon}(\cdot)$ where $\theta$ is a radially-symmetric mollifier supported on $B(0,1)$ without loss of generality. We have the following estimate controlling the difference between our Brownian motion $\overline{X}(e^{-t})$ and its approximation $\overline{X}_{\epsilon}(e^{-t})$.

\begin{lem}\label{lem:BM_eps}
Fix $\delta \in (0,1)$. There exists a sequence of random variables $\widetilde{C}_{\epsilon}$ such that
\begin{align*}
\sup_{t \le (-\log \epsilon)^{\delta}} |\overline{X}_{\epsilon}(e^{-t}) - \overline{X}(e^{-t})|
= \sup_{t \le (-\log \epsilon)^{\delta}} |B_{\epsilon, t} - B_t| 
\le \widetilde{C}_{\epsilon}
\end{align*}

\noindent and $\EE\left[\exp \left(a \widetilde{C}_{\epsilon}\right)\right] \xrightarrow{\epsilon \to 0^+} 1$ for any $a = O(\epsilon^{-1/8})$. In particular 
\begin{align*}
\lim_{\epsilon \to 0^+} \PP \left( \widetilde{C}_{\epsilon} > \epsilon^{\frac{1}{16}} \right) = 0.
\end{align*}
\end{lem}

\begin{proof}
Since $\left(\overline{X}(e^{-t})\right)_{t \ge 0}$ is a standard Brownian motion, which is e.g. $\frac{1}{3}$-H\"older continuous with stationary and independent increments, there exists a collection of i.i.d. random variables $C_{i}$ such that
\begin{align*}
|\overline{X}(e^{-t}) - \overline{X}(e^{-s})| \le C_{i} |t-s|^{\frac{1}{3}} \qquad \forall t, s \in [i, i+1].
\end{align*}

\noindent and $\EE\left[\exp(a C_i)\right] < \infty$ for any $a > 0$ by \Cref{lem:contGP}. In particular $C_i$ has positive moments of all orders.

Now consider
\begin{align*}
\overline{X}_{\epsilon}(e^{-t}) - \overline{X}(e^{-t})
= \int_{B(0,1)} \left[\overline{X}_{}(e^{-t}e_1 + \epsilon u) - \overline{X}(e^{-t})\right] \theta(u)du
\end{align*}

\noindent where $e_1$ is the first standard basis vector. Note that
\begin{align*}
\left|-\log |e^{-t} e_1 + \epsilon u| + \log |e^{-t}|\right|
\le \log |1+ \epsilon e^t|
\le \epsilon e^{(-\log \epsilon)^\delta} \le \epsilon^{3/4}
\end{align*}

\noindent for $\epsilon > 0$ sufficiently small. We see that the two numbers $t$ and $-\log |e^{-t} e_1 + \epsilon u|$ must lie in some interval of the form $[i, i+2]$, and thus
\begin{align*}
\sup_{t \le (-\log \epsilon)^{\delta}} |\overline{X}_{\epsilon}(e^{-t}) - \overline{X}(e^{-t})| \le \widetilde{C}_{\epsilon}
& \le \sum_{i=0}^{(-\log \epsilon)^{\delta}} \sup_{t \in [i, i+1]} |\overline{X}_{\epsilon}(e^{-t}) - \overline{X}(e^{-t})|\\
& \le 2 \epsilon^{1/4} \sum_{i=0}^{(-\log \epsilon)^{\delta}+1} C_i =: \widetilde{C}_{\epsilon}.
\end{align*}

\noindent We then verify, for any positive $a = O(\epsilon^{-1/8})$, that
\begin{align*}
\EE \left[ \exp\left(a \widetilde{C}_{\epsilon}\right)\right]
& = \EE \left[ \exp\left(2a \epsilon^{1/4} C_i\right)\right]^{(-\log \epsilon)^{\delta}+2}\\
& = \left(1+2a\epsilon^{1/4} \EE[C_i] + O(a^2\epsilon^{1/2})\right)^{(-\log \epsilon)^{\delta}+2}
= 1+ o(1)
\end{align*}

\noindent as $\epsilon$ tends to $0$. In particular,
\begin{align*}
\PP\left(\widetilde{C}_\epsilon > \epsilon^{1/16} \right) 
=  \PP\left(e^{\epsilon^{-1/8} \widetilde{C}_\epsilon} > e^{\epsilon^{-1/16}} \right) 
\le e^{-\epsilon^{-1/16}} \EE \left[e^{\epsilon^{-1/8} \widetilde{C}_\epsilon}\right] \xrightarrow{\epsilon \to 0^+} 0
\end{align*}

\noindent which concludes our proof.
\end{proof}

Next, we state a crucial estimate that will allow us to restrict ourselves to the leading order event $\left\{ \sup_{s \in [-\log r, -\log \epsilon]} B_s = O(1)\right\}$ as in the analysis of the toy model.

\begin{lem}\label{lem:error_con}
Let $r, \lambda > 0$ be fixed. For each $k \in \NN \cup \{0\}$ define the event
\begin{align*}
E_{\epsilon, k} := \left\{\sup_{s \in [-\log r, -\log \epsilon]} B_s - B_{-\log r} \in [k, k+1] \right\}.
\end{align*}

\noindent Then there exists $C>0$ independent of $k$ and $v \in \overline{A}$ such that
\begin{align*}
\sup_{\epsilon \in (0, \frac{r}{2}]} \left( \log \frac{1}{\epsilon}\right)^{1/2}\EE \left[ e^{-\lambda \mu_{d , \epsilon}(v, A)} 1_{E_{ \epsilon, k}}\right] \le C\sqrt{k} e^{-\frac{\lambda}{2} \sqrt{2d} k}.
\end{align*}

\noindent In particular,
\begin{align*}
\sup_{\epsilon \in (0, \frac{r}{2}]} \left( \log \frac{1}{\epsilon}\right)^{1/2}\EE \left[ e^{-\lambda \mu_{d , \epsilon}(v, A)} 1_{\{\sup_{s \in [-\log r, -\log \epsilon]} B_s \ge k \}}\right] \le C\sqrt{k} e^{-\frac{\lambda}{2}\sqrt{2d} k}.
\end{align*}
\end{lem}

\begin{proof}[Sketch of proof]
We may take $v \in \partial A$ (since the intersection $A \cap (v+A) = B(0, 2r) \cap B(v, 2r)$ is smallest if $v \in \partial A$) to obtain an upper bound that is uniform in $v \in \overline{A}$. Fix some $\delta \in (0, 1)$. We have
\begin{align*}
\EE \left[ e^{-\lambda \mu_{d , \epsilon}(v, A)} 1_{E_{\epsilon, k}}\right]
&\le \EE \left[ \mu_{d , \epsilon}(v, A)^{-\lambda} 1_{E_{\epsilon, k}}\right]\\
& \le \EE \left[ \left(e^{\sqrt{2d} B_{\epsilon, -\log r}}\int_{-\log r}^{(-\log \epsilon)^{\delta}} e^{\sqrt{2d} (B_{\epsilon, t}-B_{\epsilon, -\log r})}Z_{d, \epsilon}^{A, v}(dt) \right)^{-\lambda}1_{E_{\epsilon, k}}\right]\\
& \le \EE \left[e^{\sqrt{2d} \lambda (\widetilde{C}_{\epsilon} - B_{-\log r})} \left(\int_{-\log r}^{(-\log \epsilon)^{\delta}} e^{\sqrt{2d} (B_{t}-B_{-\log r})}Z_{d, \epsilon}^{A, v}(dt) \right)^{-\lambda}1_{E_{\epsilon, k}}\right]
\end{align*}

\noindent where $\widetilde{C}_{\epsilon}$ is as in \Cref{lem:BM_eps}. By Cauchy-Schwarz we only need to consider
\begin{align*}
\EE \left[\left(\int_{-\log r}^{(-\log \epsilon)^{\delta}} e^{\sqrt{2d} (B_{t}-B_{-\log r})}Z_{d, \epsilon}^{A, v}(dt) \right)^{-\lambda}1_{E_{\epsilon, k}}\right],
\end{align*}

\noindent which can be studied using the same method in the proof of \cite[equation (6.3)]{DKRV2017}, leading to the bound
\begin{align*}
\left(\log \frac{1}{\epsilon}\right)^{1/2}\EE \left[\left(\int_{-\log r}^{(-\log \epsilon)^{\delta}} e^{\sqrt{2d} (B_{t}-B_{-\log r})}Z_{d, \epsilon}^{A, v}(dt) \right)^{-\lambda}1_{E_{\epsilon, k}}\right] \le C k e^{-\lambda \sqrt{2d} k}
\end{align*}

\noindent and hence our claim.
\end{proof}

The final estimate we need quantifies the claim that the Brownian motion stays very negative after reaching the maximum.
\begin{lem}\label{lem:BM_neg}
Let $t = -\log \epsilon$ and fix $k > 0$. As $t \to \infty$ we have
\begin{align*}
\PP\left(\max_{s \in [t^{1/2}, t]} B_s - B_{-\log r} \ge k - t^{1/8}  \bigg| \max_{s \in [-\log r, t]} B_s - B_{-\log r} \le k \right) = o(1).
\end{align*}
\end{lem}

\begin{proof}
For simplicity let us only treat $r = 1$ (and hence $B_{-\log r} = 0$) everything below works for any $r > 0$. Using the fact that $\PP(\max_{s \le t} B_s \le n) = \sqrt{\frac{2}{\pi}}\int_0^{n/\sqrt{t}} e^{-x^2 / 2}dx$, we first obtain
\begin{align*}
& \PP\left(\max_{s \in [t^{1/2}, t]} B_s \in [k - t^{1/8}, k] \bigg| B_{t^{1/2}} \right)\\
& \qquad = \PP\left(\max_{s \in [t^{1/2}, t]} B_s - B_{t^{1/2}} \in [k - t^{1/8} -B_{t^{1/2}}, k - B_{t^{1/2}}] \bigg| B_{t^{1/2}} \right)\\
& \qquad \le \sqrt{\frac{2}{\pi (t - t^{1/2})}} \left[(k - B_{t^{1/2}})_+ 1_{\{k - B_{t^{1/2}} \le t^{1/8}\}} + t^{1/8} 1_{\{k - B_{t^{1/2}} \ge t^{1/8}\}}\right].
\end{align*}

\noindent Then
\begin{align*}
& \PP\left(\max_{s \in [t^{1/2}, t]} B_s \ge k - t^{1/8} , \max_{s \in [0, t]} B_s \le k \right)\\
& \qquad = \EE \left[\PP\left(\max_{s \in [t^{1/2}, t]} B_s \in [k - t^{1/8}, k] \bigg| B_{t^{1/2}} \right) 1_{\{\max_{s \in [0, t^{1/2}]} B_s \le k\}} \right]\\
& \qquad \le \sqrt{\frac{2}{\pi (t - t^{1/2})}} \EE \left[ \left(k -B_{t^{1/2}}\right)_+ 1_{\{\max_{s \in [0, t^{1/2}]} B_s \le k\}} 1_{\{k - B_{t^{1/2}} \le t^{1/8}\}}\right]\\
& \qquad \qquad + \sqrt{\frac{2}{\pi (t - t^{1/2})}} \underbrace{t^{1/8} \PP\left(\max_{s \le t^{1/2}} B_s \le k\right)}_{=O(t^{-1/8})}.
\end{align*}

To finish our proof, we only have to show that
\begin{align*}
& \EE \left[ \left(k -B_{t^{1/2}}\right)_+ 1_{\{\max_{s \in [0, t^{1/2}]} B_s \le k\}} 1_{\{k - B_{t^{1/2}} \le t^{1/8}\}}\right] \\
& = \EE \left[ \left(k -B_{t^{1/2}}\right)_+ 1_{\{\max_{s \in [0, t^{1/2}]} B_s \le k\}}\right] 
\EE \left[ \frac{ \left(k -B_{t^{1/2}}\right)_+ 1_{\{\max_{s \in [0, t^{1/2}]} B_s \le k\}} 1_{\{k - B_{t^{1/2}} \le t^{1/8}\}}}{\EE \left[ \left(k -B_{t^{1/2}}\right)_+ 1_{\{\max_{s \in [0, t^{1/2}]} B_s \le k\}}\right]}\right] \\
& = o(1).
\end{align*}

\noindent But then by \Cref{lem:RadNik}, this is just equal to $k$ multiplied by the probability that a $\mathrm{BES}_k(3)$-process at time $t^{1/2}$ is less than $t^{1/8}$, which obviously tends to $0$ as $t \to \infty$.
\end{proof}

\subsection{Sketch of proof of \Cref{lem:fusion}} 
Let us write $\overline{E}_{\epsilon, k} = \cup_{j < k} E_{\epsilon, j} =  \left\{\sup_{s \in [-\log r,t]} B_s - B_{-\log r} \le k \right\}$  and introduce
\begin{align*}
\Ga_{\epsilon, k} := \left\{\max_{s \in [t^{1/2}, t]} B_s - B_{-\log r} \le k - t^{1/8}\right\}
\end{align*}

\noindent where $t = -\log \epsilon$ as in \Cref{lem:BM_neg}. Then, 
\begin{align*}
& \left(\log \frac{1}{\epsilon}\right)^{1/2} \EE\left[e^{-\lambda \mu_{d, \epsilon}(v, A)}\right]
= \left(\log \frac{1}{\epsilon}\right)^{1/2} \EE\left[e^{-\lambda \mu_{d, \epsilon}(v, A)} 1_{\Ga_{\epsilon, k}} 1_{\overline{E}_{\epsilon, k}}\right]\\
& \qquad \qquad + \left(\log \frac{1}{\epsilon}\right)^{1/2} O\left(\PP\left(\Ga_{\epsilon, k}^c \cap \overline{E}_{\epsilon, k}\right)\right) + \left(\log \frac{1}{\epsilon}\right)^{1/2} O \left(\EE\left[ e^{-\lambda \mu_{d, \epsilon}(v, A)} 1_{E_{\epsilon, k}^c}\right]\right)
\end{align*}

\noindent where the first error is of order $o(1)$ as $\epsilon \to 0^+$ (depending on $k$) by \Cref{lem:BM_neg}, and the second error is of order $O(\sqrt{k} e^{-\frac{\lambda}{2} \sqrt{2d}k})$ uniformly in $\epsilon \to 0^+$ by \Cref{lem:error_con}. In other words we only need to focus on
\begin{align}\label{eq:fusion_cmain}
\left(\log \frac{1}{\epsilon}\right)^{1/2} \EE\left[e^{-\lambda \mu_{d, \epsilon}(v, A)} 1_{\Ga_{\epsilon, k}} 1_{\overline{E}_{\epsilon, k}}\right].
\end{align}

Recall from \Cref{subsec:fusion_idea} that $\mu_{d, \epsilon}(v, A) = \mu_{d, \epsilon}(v, r) + \mu_{d, \epsilon}^c(v) + R_{d, \epsilon}(v, r)$. The first observation here is that $R_{d, \epsilon}(v, r)$ is essentially\footnote{It only depends on $(B_{ \epsilon, s})_{s \le -\log r}$, which is essentially $(B_s)_{s \le -\log r}$ up to a vanishing error by \Cref{lem:BM_eps}.} independent of $\overline{E}_{\epsilon, k}$, and so it still converges to $R_d(v, r)$ even when conditioned on this sequence of events as $\epsilon \to 0^+$. On the good event $\Ga_{\epsilon, k}$, we also see that
\begin{align*}
\mu_{d, \epsilon}(v, r)
& = e^{\sqrt{2d} B_{\epsilon, -\log r}}\int_{-\log r}^{-\log 2\epsilon} e^{\sqrt{2d} (B_{\epsilon, s}-B_{\epsilon, -\log r})}Z_{d, \epsilon}^{A, v}(ds),\\
& = e^{\sqrt{2d} (B_{-\log r} + O(\widetilde{C}_{\epsilon}))}
\Bigg[\int_{-\log r}^{t^{1/2}} e^{\sqrt{2d} (B_{s}-B_{-\log r})}Z_{d, \epsilon}^{A, v}(ds) + O\left(e^{\sqrt{2d} (k - t^{1/8})}\int_{t^{1/2}}^{t-\log 2} Z_{d, \epsilon}^{A, v}(ds)\right)\Bigg],\\
\mu_{d, \epsilon}^c(v)
& \le 4^d e^{\sqrt{2d} B_{\epsilon, -\log 2\epsilon}} \overline{\mu}_{d, \epsilon}^c(v) 
\le 4^d e^{\sqrt{2d} (k - t^{1/8} )} \overline{\mu}_{d, \epsilon}^c(v).
\end{align*}

\noindent But since
\begin{align*}
\left(\log \frac{1}{\epsilon}\right)^{1/2} \EE\left[1_{\overline{E}_{\epsilon, k}}\right]
\le k \sqrt{\frac{2}{\pi}} \sqrt{\frac{t}{t + \log r}}
\end{align*}

\noindent is bounded uniformly in $\epsilon \to 0^+$ (or equivalently $t \to \infty$) for each fixed $k, r > 0$, we may ignore $\mu_{d, \epsilon}^c(v)$ and the residual term $O\left(e^{\sqrt{2d} (k - t^{1/8})}\int_{t^{1/2}}^{t-\log 2}Z_{d, \epsilon}^{A, v}(ds)\right)$ in $\mu_{d, \epsilon}(v, r)$ and argue by dominated convergence that \eqref{eq:fusion_cmain} is asymptotically equal to  
\begin{align*}
& (1+o(1))\left(\log \frac{1}{\epsilon}\right)^{1/2} \EE\left[\frac{1_{\Ga_{\epsilon, k}}1_{\overline{E}_{\epsilon, k}}}{\exp\left(\lambda \left(e^{\sqrt{2d} B_{-\log r}}\int_{-\log r}^{t^{1/2}} e^{\sqrt{2d} (B_s - B_{-\log r})}Z_{d, \epsilon}^{A, v}(ds) + R_d(v, r)\right)\right)}\right]\\
& = (1+o(1)) \left(\log \frac{1}{\epsilon}\right)^{1/2}\EE\left[\frac{1_{\overline{E}_{\epsilon, k}}}{\exp\left(\lambda \left(e^{\sqrt{2d} B_{-\log r}}\int_{-\log r}^{t^{1/2}} e^{\sqrt{2d} (B_s - B_{-\log r})}Z_{d, \epsilon}^{A, v}(ds) + R_d(v, r)\right)\right)}\right] \\
& \qquad \qquad +  \underbrace{O\left(\left(\log \frac{1}{\epsilon}\right)^{1/2}\PP\left(\Ga_{\epsilon, k}^c \cap \overline{E}_{\epsilon, k}\right)\right)}_{=o(1)}.
\end{align*}

Let us write $(W_s)_{s \ge 0} = (B_{s - \log r} - B_{- \log r})_{s \ge 0}$, which is a Brownian motion independent of $B_{-\log r}$, and denote by $(\Fa_s)_{s \ge 0}$ its natural filtration. Using again the distribution for the running maximum of a Brownian motion, we have
\begin{align*}
& \EE\left[1_{\overline{E}_{\epsilon, k}} \bigg| \Fa_{t^{1/2} + \log r}\right]
= \PP\left(\max_{s \in [0, t + \log r]} W_s \le k \bigg| \Fa_{t^{1/2} + \log r}\right)\\
& = 1_{\{\max_{s \le t^{1/2} + \log r} W_s \le k\}} \PP\left(\max_{s \in [t^{1/2} + \log r, t + \log r]} W_s \le k \bigg| \Fa_{t^{1/2} + \log r}\right)\\
&\begin{dcases}
\le 1_{\{\max_{s \le t^{1/2} + \log r} W_s \le k\}} \sqrt{\frac{2}{\pi}}  \frac{k - W_{t^{1/2} + \log r}}{\sqrt{t - t^{1/2}}}, \\
\ge  1_{\{\max_{s \le t^{1/2} + \log r} W_s \le k\}} \sqrt{\frac{2}{\pi}}  \frac{k - W_{t^{1/2} + \log r}}{\sqrt{t - t^{1/2}}}e^{-(t-t^{1/2})^{1/4}} 1_{\{(k - W_{t^{1/2} + \log r})^2 \le 2(t-t^{1/2})^{3/4}\}}.
\end{dcases}
\end{align*}

Interpreting the process $(k-W_s)_{s \le t^{1/2} + \log r}$ as a $\mathrm{BES}_k(3)$-process $(\beta_s^k)_{s \ge 0}$ under the change of measure
\begin{align*}
\frac{1_{\{\max_{s \le t^{1/2} + \log r} W_s \le k\}}(k - W_{t^{1/2} + \log r})}{\EE\left[1_{\{\max_{s \le t^{1/2} + \log r} W_s \le k\}}(k - W_{t^{1/2} + \log r})\right]} d\PP 
= \frac{1}{k} 1_{\{\max_{s \le t^{1/2} + \log r} W_s \le k\}}(k - W_{t^{1/2} + \log r})  d\PP
\end{align*}

\noindent by \Cref{lem:RadNik}, we see that
\begin{align*}
& \EE\left[1_{\{\max_{s \le t^{1/2} + \log r} W_s \le k\}} (k - W_{t^{1/2} + \log r}) 1_{\{(k - W_{t^{1/2} + \log r})^2 > 2(t-t^{1/2})^{3/4}\}}\right]\\
& \qquad = k \PP\left( \beta_{t^{1/2} + \log r}^k > 2(t-t^{1/2})^{3/4}\right) = o(1), \qquad t \to \infty
\end{align*}

\noindent and hence

\begin{align*}
& \left(\log \frac{1}{\epsilon}\right)^{1/2} \EE\left[\frac{1_{\overline{E}_{\epsilon, k}}}{\exp\left(\lambda \left(e^{\sqrt{2d} B_{-\log r}}\int_{-\log r}^{t^{1/2}} e^{\sqrt{2d} (B_s - B_{-\log \epsilon})}Z_{d, \epsilon}^{A, v}(ds) + R_d(v, r)\right)\right)}\right]\\
& =  (1+o(1)) \sqrt{\frac{2}{\pi}} \EE\left[\frac{1_{\{\max_{s \le t^{1/2} + \log r} W_s \le k\}}(k - W_{t^{1/2} + \log r})}{\exp\left(\lambda \left(e^{\sqrt{2d} B_{-\log r}}\int_{0}^{t^{1/2} + \log r} e^{\sqrt{2d}W_{s}}Z_{d, \epsilon}^{A, v} \circ \phi_{-\log r} (ds) + R_d(v, r)\right)\right)}\right]\\
& =  (1+o(1)) \sqrt{\frac{2}{\pi}} k \EE\left[\exp\left(-\lambda \left(e^{\sqrt{2d} B_{-\log r}}\int_{0}^{t^{1/2} + \log r} e^{\sqrt{2d}(k - \beta_s^k)}Z_{d, \epsilon}^{A, v} \circ \phi_{-\log r} (ds) + R_d(v, r)\right)\right)\right].
\end{align*}

Summarising all the analysis above, we have
\begin{align}
\notag
&  \left(\log \frac{1}{\epsilon}\right)^{1/2} \EE\left[e^{-\lambda \mu_{d, \epsilon}(v, A)}\right]\\
\notag 
& \qquad \xrightarrow{\epsilon \to 0^+}\sqrt{\frac{2}{\pi}} k \EE\left[\exp\left(-\lambda \left(e^{\sqrt{2d} B_{-\log r}}\int_{0}^{\infty} e^{\sqrt{2d}(k - \beta_s^k)}Z_d^{A, v} \circ \phi_{-\log r} (ds) + R_d(v, r)\right)\right)\right]\\
\label{eq:fusion_almost}
& \qquad \qquad \qquad + O(\sqrt{k} e^{-\frac{\lambda}{2} \sqrt{2d}k}).
\end{align}

We now apply \Cref{lem:BES_path}, which allows us to rewrite $(k - \beta_s^k)_{s \ge 0}$ as the new process
\begin{align*}
R_s = 
\begin{cases}
W_s, & s \le \tau(\widetilde{U}), \\
\widetilde{U} + \beta_{s - \tau(\widetilde{U})}^0, & s \ge  \tau(\widetilde{U})
\end{cases}
\end{align*}

\noindent where $(W_s)_{s \ge 0}$ is a standard Brownian motion, $\widetilde{U}$ is a $\mathrm{Uniform}[0, k]$ random variable, $\tau(\widetilde{U}) := \inf \{s > 0: W_s = \widetilde{U}\}$, and $(\beta_s^0){s \ge 0}$ is a $\mathrm{BES}_0(3)$-process. Using \Cref{lem:time_rev}, we may further write
\begin{align*}
(R_{s + \tau(\widetilde{U})})_{s \in \left[-\tau(\widetilde{U}), \infty\right]} \overset{d}{=} (\widetilde{U} + \beta_s)_{s \in \left[-L_{\widetilde{U}, -}, \infty\right]}
\end{align*}

\noindent where $(\beta_s)_{s \ge 0}$ and $(\beta_{-s})_{s \ge 0}$ are two independent $\mathrm{BES}_0(3)$-processes and
\begin{align*}
L_{x, -} := \sup \{ s \ge 0: \beta_{-s} = x \}.
\end{align*}

\noindent This gives
\begin{align*}
& k \EE\left[\exp\left(-\lambda \left(e^{\sqrt{2d} B_{-\log r}}\int_{0}^{\infty} e^{\sqrt{2d}(k - \beta_s^k)}Z_d^{A, v} \circ \phi_{-\log r} (ds) + R_d(v, r)\right)\right)\right]\\
& = k \EE\left[\exp\left(-\lambda \left(e^{\sqrt{2d} B_{-\log r}}\int_{-L_{\widetilde{U}, -}}^{\infty} e^{\sqrt{2d}(\widetilde{U}- \beta_s)}Z_d^{A, v} \circ \phi_{-\log r + L_{\widetilde{U}, -}} (ds) + R_d(v, r)\right)\right)\right]\\
& = \int_0^k \EE\left[\exp\left(-\lambda \left(e^{\sqrt{2d}x} \widetilde{\mu}_d^x(v, r) + R_d(v, r)\right)\right)\right] dx
\end{align*}

\noindent where 
\begin{align*}
\widetilde{\mu}_d^x (v, r) = e^{\sqrt{2d}B_{-\log r}} \int_{-L_{x, -}}^\infty e^{-\sqrt{2d}\beta_s} Z_d^{A, v} \circ \phi_{- \log r + L_{x, -}}(ds).
\end{align*}

\noindent Plugging this into \eqref{eq:fusion_almost} and sending $k \to \infty$, we conclude that
\begin{align*}
\lim_{\epsilon \to 0^+} \left(\log \frac{1}{\epsilon}\right)^{1/2} \EE\left[e^{-\lambda \mu_{d, \epsilon}(v, A)}\right]
& = \sqrt{\frac{2}{\pi}} \int_0^\infty \EE\left[\exp\left(-\lambda \left(e^{\sqrt{2d}x} \widetilde{\mu}_d^x(v, r) + R_d(v, r)\right)\right)\right] dx.
\end{align*}
\hfill \qed

\section{Beyond Jordan measurable sets} \label{app:jordan}
\subsection{Issues with general open sets}
Revisiting \Cref{lem:universal},  the only place where Jordan measurability is needed is the derivation of an estimate of the form
\begin{align}\label{eq:error_need}
\PP\left(\mu_{f}(B_\delta) > t \right) \le \frac{C \delta}{t} 
\qquad \text{ or } \qquad
\lambda^{-1/2}\EE \left[ 1 - e^{-\lambda \mu_f(B_{\delta})^2}\right] \le C \delta
\end{align}

\noindent for any open set $B_\delta$ such that $|B_\delta| \le \delta$. 

When $B_\delta$ is not Jordan measurable, the theory of Lebesgue integration suggests that $B_{\delta}$ can be covered by a countable union $\widetilde{B}_{\delta} = \cup_{i=1}^\infty B_{\delta, i}$ of open cubes $B_{\delta, i}$ up to a small error $\delta > 0$, i.e. $|\widetilde{B}_{\delta} \setminus B_{\delta}| \le \delta$. By further partitioning, we may assume that the interior of these cubes are disjoint from each other, and the upper bound of the splitting lemma reads as
\begin{align*}
\EE \left[ 1 - e^{-\lambda \mu_f(B_{\delta})^2}\right]
\le  \sum_i \EE \left[ 1 - e^{-\lambda \mu_f(B_{\delta, i})^2}\right]
+ \sum_{j < k} \EE \left[ 1 - e^{-2\lambda \mu_f(B_{\delta, j})\mu_f(B_{\delta, k})}\right].
\end{align*}

\noindent It is not difficult to check that there exists some $C>0$ uniformly in everything such that $\PP(\mu_f(B_{\delta, i})> t) \le C|B_{\delta, i}| / t$, from which we obtain 
\begin{align*}
\lambda^{-1/2}\sum_i \EE \left[ 1 - e^{-\lambda \mu_f(B_{\delta, i})^2}\right] \le C' \sum_{i} |B_{\delta, i}| \le 2C' \delta.
\end{align*}

\noindent The real issue is to treat the cross terms $\EE \left[ 1 - e^{-2\lambda \mu_f(B_{\delta, j})\mu_f(B_{\delta, k})}\right]$: while we know each of them is of order $o(\lambda^{1/2})$, the current estimate (for the hidden constant in the little-o notation) is too weak to allow an application of dominated convergence in order to interchange the summation and limit.

\subsection{A potential direct approach to low dimensions}
We restrict our discussion below to $d=2$ even though the same argument applies to $d = 1$. Here we take $\mu_f(dx)$ as the critical GMC associated to the exact field $Y_0$ on $D = B(0,1) \subset \RR^2 \equiv \CC$, and our goal is to show that
\begin{align*}
\PP\left(\mu_f(A) > t \right) \overset{t \to \infty}{\sim} \frac{|A|}{\sqrt{2 \pi} t}
\end{align*}

\noindent for any open $A \subset B(0,1)$. Following \cite{Won2019}, we shall consider a different Laplace estimate and aim to show that
\begin{align*}
\EE\left[e^{-\lambda / \mu_f(A)}\right] \overset{\lambda \to \infty}{\sim} \frac{|A|}{\sqrt{2\pi}\lambda}
\end{align*}

\noindent which is equivalent to the desired tail asymptotics by Tauberian \cref{theo:tau} (by taking $\nu(ds) = \PP(\mu_f(A)^{-1}\in ds)$). Using the same ideas in \Cref{app:fusion}, one expects to obtain the localisation limit
\begin{align}\label{eq:ll_direct}
\EE\left[e^{-\lambda / \mu_f(A)}\right]
= \int_A \sqrt{\frac{2}{\pi}} \left(\int_0^\infty  \EE \left[ \frac{e^{-\lambda/\left(e^{\sqrt{2d}x} \widetilde{\mu}_d^x(v, r) + R_d(v, r)\right)}}{e^{\sqrt{2d}x} \widetilde{\mu}_d^x(v, r) + R_d(v, r)} \right] dx\right) dv
\end{align}

\noindent as well as the asymptotics
\begin{align}\label{eq:le_direct}
\lim_{\lambda \to \infty} \lambda \int_0^\infty  \EE \left[ \frac{e^{-\lambda/\left(e^{\sqrt{2d}x} \widetilde{\mu}_d^x(v, r) + R_d(v, r)\right)}}{e^{\sqrt{2d}x} \widetilde{\mu}_d^x(v, r) + R_d(v, r)} \right] dx = \frac{1}{\sqrt{2d}}
\end{align}

\noindent where $d = 2$ and $\widetilde{\mu}_{d}^x(v, r)$ and $R_d(v, r)$ are defined as in \Cref{lem:fusion}. 

The only difficulty in this direct approach is to justify the interchanging of limits and integration when $A$ is not a very nice set. For \eqref{eq:ll_direct}, the key is to make sure that the constant $C$ appearing in the bound of \Cref{lem:error_con} is indeed uniform in $ v \in A$. As for \eqref{eq:le_direct}, it requires better control over moments related to $e^{\sqrt{2d}x} \widetilde{\mu}_d^x(v, r) + R_d(v, r)$ since the bounds for $\EE\left [\widetilde{\mu}_d^x(v, r)^q \right]$ in \Cref{lem:fusion_moment} are not necessarily uniform in $v \in A$ and $x \ge 0$ when $q < 0$ in the general setting.



\end{document}